%% file: MAIN.tex
\documentclass{amsart}
\usepackage[margin=.9in]{geometry}
\usepackage[utf8]{inputenc}
\usepackage[english]{babel}
\usepackage[
backend=biber,
style=alphabetic,
maxalphanames=4
]{biblatex}
\input{Structure/pre}

\title{$H$-Equivariant Morita Equivalences of Loewy-Graded Comodule Algebras}
\author{Jacob Van Grinsven}
\date{October 10, 2025}


\begin{document}

\input{Abstract}

\maketitle

\setcounter{tocdepth}{1}
\tableofcontents


\input{Sections/1.Introduction}
\input{Sections/2.Background}


\input{Sections/3.LG_Endo}

\input{Sections/4.LG_AMEx_HCMA}
\input{Sections/5.KP_Mod_Cats}

\input{Acknowledge}
\printbibliography
\end{document}

%% file: Structure/pre.tex
\input{Structure/Packages}
\input{Structure/Commands}

\input{Structure/Letters}

%% file: Structure/Commands.tex
\newcommand{\bracks}[1]{\left[#1\right]}
\newcommand{\paren}[1]{\left(#1\right)}
\newcommand{\set}[1]{\left\{#1\right\}}

\newcommand{\half}{\frac{1}{2}}
\renewcommand{\bar}{\overline}
\renewcommand{\subset}{\subseteq}

\newcommand{\tensor}{\otimes}

\DeclareMathOperator{\id}{id}

\DeclareMathOperator{\Rep}{Rep}

\DeclareMathOperator{\Hom}{Hom}
\DeclareMathOperator{\End}{End}
\DeclareMathOperator{\co}{co}

\newcommand{\gr}{\mathrm{gr}}
\newcommand{\GA}[1]{{\kk[#1]}}

\newcommand{\kK}{\GA{K}}
\newcommand{\kkpsiK}{{\kK_\psi}}

\newcommand{\hact}{\curvearrowright}
\newcommand{\yd}[1]{{}_{#1}^{#1}\mathcal{YD}}
\newcommand{\hkp}{{\mathcal{KP}}}

\numberwithin{equation}{section}

\newtheorem{theorem}{Theorem}[section]
\newtheorem{lemma}[theorem]{Lemma}
\newtheorem{proposition}[theorem]{Proposition}
\newtheorem{corollary}[theorem]{Corollary}
\newtheorem{conjecture}[theorem]{Conjecture}

\theoremstyle{definition}
\newtheorem{remark}[theorem]{Remark}
\newtheorem{definition}[theorem]{Definition}

\newtheorem{question}[theorem]{Question}

%% file: Structure/Letters.tex
\newcommand{\kk}{\Bbbk}

\newcommand{\M}{\mathcal{M}}
\newcommand{\F}{\mathcal{F}}
\newcommand{\G}{\mathcal{G}}
\newcommand{\N}{\mathcal{N}}
\newcommand{\C}{\mathcal{C}}
\newcommand{\B}{\mathfrak{B}}

\newcommand{\Z}{\mathbb{Z}}

\renewcommand{\epsilon}{\varepsilon}
\renewcommand{\phi}{\varphi}

%% file: Abstract.tex
\begin{abstract}
Let $H$ be a coradically graded Hopf algebra. For every Loewy-graded exact $H$-comodule algebra $A=\oplus_{n\geq 0} A(n)$ and $H_0$-equivariant Morita equivalence $A(0)\simeq_{H_0} X$, there exists a Loewy-graded $H$-comodule algebra $B$ (isomorphic to $X$ in degree zero) realizing an $H$-equivariant Morita equivalence $A\simeq_H B$. In addition, if every exact $H_0$-comodule algebra is $H_0$-equivariant Morita equivalent to a coideal subalgebra of $H_0$, then every Loewy-graded exact $H$-comodule algebra is $H$-equivariant Morita equivalent to a coideal subalgebra of $H$. We also discuss Loewy-graded $H$-comodule algebras with $H_0=\hkp$, the Kac-Paljutkin Hopf algebra.
\end{abstract}

%% file: Sections/1.Introduction.tex
\section{Introduction}
A finite-dimensional Hopf algebra $H$ may be viewed as a generalization of a finite group, in the sense that the category of finite-dimensional (left) $H$-modules, ${}_H\M$, forms a tensor category, as in \cite{EGNO}. In this manner, the category ${}_H\M$ admits additive and multiplicative structures analogous to those of a ring. This motivates the study of the structure of ${}_H\M$ though the study of ${}_H\M$-module categories, a categorical analogue of a module over a ring.


Exact ${}_H\M$-module categories have been shown to be parameterized by exact algebras in ${}_H\M$ (see \cite{EGNO}, Corollary 7.10.5). The study of $H$-module algebras and Hopf actions on finite-dimensional algebras is an active area of research. A (very) small selection of works regarding $H$-module algebras (and their associated module categories) includes \cite{CEW_Weyl}, \cite{CKWZ16}, \cite{CE17}, \cite{KO21}, and \cite{EKW}.


It was shown in \cite{AM7} that, in characteristic zero, ${}_H\M$-module categories are parameterized by equivalence classes of $H$-comodule algebras. In particular, each ${}_H\M$-module category is equivalent to the category of finite-dimensional $A$-modules for some $H$-comodule algebra $A$ with trivial coinvariants and simple in ${}^H\M_A$. We call such comodule algebras \emph{AM-exact}. This is in contrast with the description of ${}_H\M$-module categories associated to $H$-module algebras, in which the module category is defined as the category of $A$-modules ``inside" of ${}_H\M$ (see Definition 7.8.5 in \cite{EGNO}). Another benefit of this perspective is that, for any inclusion of Hopf algebras $H\subset H'$, every $H$-comodule algebra is naturally an $H'$-comodule algebra. That is, if $A$ is a left $H$-comodule algebra, then ${}_A\M$ naturally inherits the structure of both an ${}_H\M$-module category and an ${}_{H'}\M$-module category.

The correspondence between ${}_H\M$-module categories and $H$-comodule algebras is not unique, in the sense that multiple $H$-comodule algebras may generate the same ${}_H\M$-module category. This draws direct connections to classical Morita theory. If $A$ and $B$ generate equivalent ${}_H\M$-module categories, then we say $A$ and $B$ are \emph{$H$-equivariant Morita equivalent}. A set $X$ of $H$-comodule algebras is a \emph{complete set of exact $H$-equivariant Morita equivalence class representatives} if every exact indecomposable ${}_H\M$-module category is equivalent to the category of $A$-modules for some $A\in X$.


A method inspired by the lifting method (see \cite{LiftQuantLin}, \cite{ANP}) was developed in \cite{MPT} to classify ${}_H\M$-module categories via $H$-equivariant equivalence classes of comodule algebras. Explicit computations of $H$-equivariant Morita equivalence classes of AM-exact comodule algebras for several pointed Hopf algebras with cyclic coradical (e.g. Taft algebras, $u_q(\mathfrak{sl}_2)$) were also given in \cite{MPT}. Motivated by this work, we make the following conjecture:
\begin{conjecture}\label{Conj 1}
    Let $H$ be a finite-dimensional Hopf algebra with the dual Chevalley property, $H_0$ its coradical, and $X_0$ a complete set of exact $H_0$-equivariant Morita equivalence class representatives. Then the set
    \begin{equation*}
        X:=\set{A: \lambda^{-1}(H_0\tensor A)\text{ is isomorphic to an element in }X_0}
    \end{equation*}
    is a complete set of exact $H$-equivariant Morita equivalence class representatives.
\end{conjecture}
A proof of Conjecture \ref{Conj 1} would imply that it is sufficient to consider lifts (in the sense of \cite{MPT}) of $H_0$-equivariant Morita equivalence class representatives. We prove Conjecture \ref{Conj 1} along the first step in the comodule algebra lifting method. In particular, we prove the following theorem.
\begin{theorem}\label{Main Thm s1}
    Let $H$ be a coradically graded Hopf algebra. If $X_0$ is a complete set of AM-exact $H_0$-equivariant Morita equivalence class representatives, then every Loewy-graded, AM-exact $H$-comodule algebra $A$ is $H$-equivariant Morita equivalent to an element of
    \begin{equation*}
        X:=\set{A: \text{Loewy-graded with } A(0)\text{ isomorphic to an element in }X_0}.
    \end{equation*}
\end{theorem}


We discuss tensor categories, Hopf algebras, and other related algebraic structures in Section 2. Section 3 is dedicated to describing the endomorphism algebra of Loewy-graded $A$-modules. These results are central to Section 4 where we prove Theorem \ref{Main Thm s1}. Lastly, Section 5 is dedicated to showing that every ${}_{\hkp}\M$-module category is equivalent to ${}_A\M$ for some coideal subalgebra $A\subset \hkp$ where $\hkp$ is the Kac-Paljutkin Hopf algebra. We also discuss the Loewy-graded $H$-comodule algebras when $H$ is coradically graded and $H_0=\hkp$.

%% file: Sections/2.Background.tex
\section{Background and preliminaries}
Throughout, $\kk$ denotes an algebraically closed field of characteristic zero and $H$ a finite-dimensional Hopf algebra. We denote the comultiplication on $H$ by $\Delta$ and counit $\epsilon$. For a finite-dimensional $\kk$-algebra $A$, the category of finite-dimensional left $A$-modules will be denoted ${}_A\M$. Similarly, $\M_A$ denotes the category of finite-dimensional right $A$-modules.

\subsection{Tensor and Module Categories}

A \textit{tensor category} $\C$ is a $\kk$-linear abelian rigid monoidal category with simple monoidal unit $\mathbb{1}_\C$ such that the bifunctor $\tensor_\C:\C\times \C\to \C$ is bilinear on morphisms. Let $a$ be the associativity constraint and $r,\ell$ the right and left unit constraints respectively. One example of particular interest is, ${}_H\M$, the category of finite-dimensional left $H$-modules for a finite-dimensional Hopf algebra $H$. In particular, if $H$ is a finite-dimensional Hopf algebra, then the tensor product (over $\kk$) of representations 
$V,W\in {}_H\M$ naturally inherits the structure of an $H\tensor_\kk H$-module. The coproduct is a morphism of algebras that defines an $H$-module structure on $V\tensor_\kk W$. With this definition of the tensor product of $H$-modules, ${}_H\M$ has the structure of a tensor category. For a complete proof of this fact, along with a comprehensive treatment of tensor (and module categories), we direct the reader to \cite{EGNO}.

A left $\C$-\textit{module category} is an abelian category $\M$, equipped with an exact bifunctor $\tensor_\M:\C\times \M\to \M$ and natural isomorphisms:
\begin{align*}
    m_{X,Y,M}&:(X\tensor_\C Y)\tensor_\M M\to X\tensor_\M (Y\tensor_\M M)\\
    \ell_M&:\mathbb{1}_\C\tensor_\M M\to M
\end{align*}
such that for all $X,Y,Z\in \C, M\in \M$:
\begin{align*}
    (\id_X\tensor m_{Y,Z,M})m_{X,Y\tensor Z,M}(a_{X,Y,Z}\tensor \id_M)&=m_{X,Y,Z\tensor M}m_{X\tensor Y,Z,M},\\
    (\id_X\tensor \ell_M)m_{X,\mathbb{1}_\C,M}.
\end{align*}
Given two left $\C$-module categories $\M,\M'$, a \textit{$\C$-module functor} is a pair $(\F,c)$ such that $\F$ is a $\kk$-linear functor $\F:\M\to \N$ and $c$, a natural isomorphism \begin{align*}
    c_{X,M}:\F(X\tensor_\M M)\to X\tensor_{\N} \F(M),
\end{align*}
such that for all $X,Y\in \C, M\in \M$:
\begin{align*}
    (\id_X\tensor c_{Y,M})c_{X,Y\tensor M}\F(m_{X,Y,M})&=m_{X,Y,\F(M)}c_{X\tensor Y,M}\\
    \ell_{\F(M)}c_{\mathbb{1}_\C,M}&=\F(\ell_M).
\end{align*}
We can see that the composition of $\C$-module functors is a $\C$-module functor. In particular, $(\G,d)\circ (\F,c)=(\G\circ \F, d\circ \G(c))$. A \textit{natural transformation of $\C$-module functors} $(\F,c),(\G,d):\M\to \N$ is a natural transformation $\eta:\F\to\G$ such that
\begin{equation*}
    d_{X,M}\eta_{X\tensor M}=(\id_X\tensor \eta_M)c_{X,M}
\end{equation*}
for all $X\in\C$ and $M\in \M$. We say that $\M$ and $\N$ are \textit{equivalent} as $\C$-module categories if there exists functors $F:\M\to \N$, $\G:\N\to \M$ and natural isomorphisms of $\C$-module functors $\id_\N\to \G\circ \F$ and $\id_\N\to \F\circ \G$.

The direct sum of two $\C$-module categories $\M$ and $\N$ is defined by the coordinate wise action on the product category $\M\times \N$. We say that $\M$ is indecomposable if it is not equivalent as a $\C$-module category to the direct sum of two (non-trivial) $\C$-module categories. From a representation theoretic perspective, it is natural to study $\C$ by studying $\C$-module categories. The notion of an exact $\C$-module category was developed in \cite{FinTen} and provides a categorical analogue to a projective module.

\begin{definition}[Exact Module Categories]
    Let $\C$ be a finite tensor category. We say the $\C$-module category $\M$ is \textit{exact} if for all projective $P\in\C$, and arbitrary $M\in \M$, $P\tensor M\in\M$ is projective.
\end{definition}
Notice if $\M$ semisimple, then every object in $\M$ is projective so $\M$ is exact. Moreover, if $\C$ is fusion then every exact module category is semisimple (see Example 7.5.4 in \cite{EGNO}) but module categories can be quite complicated when $\C$ (or $H$) is not semisimple.

\subsection{$H$-Comodule Algebras}
When $\C={}_H\M$ for a finite dimensional Hopf algebra $H$, we can construct module categories from left $H$-comodule algebras. To begin, we recall several constructions.

A left $H$-comodule algebra is a $\kk$-algebra $A$ equipped with a coaction $\lambda_A:A\to H\tensor A$ such that $\lambda_A$ is a morphism of $\kk$-algebras. We use (sumless) Sweedler notation,
\begin{equation*}
    \lambda_A(a)=a_{(-1)}\tensor a_{(0)}
\end{equation*}
and use $\lambda=\lambda_A$ when no confusion can arise. We will see shortly that the category of finite-dimensional representations of $A$, denoted ${}_A\M$, admits the structure of a left ${}_H\M$-module category. The coinvariant subspace, $A^{\co H}=\set{a\in A:\lambda(a)=1\tensor a}$ is an $H$-costable subalgebra of $A$. The unit of $A$ is always contained in the coinvariant subalgebra $A^{\co H}$ and if $\dim(A^{\co H})=1$ then we say $A$ has trivial coinvariants. We say $A$ is right $H$-simple if there are no proper ideals $J\subset A$ such that $\lambda(J)\subset H\tensor J$.

We recall the construction in \cite{AM7} for the action of ${}_H\M$ on ${}_A\M$: For any $X\in {}_H\M, M\in {}_A\M$, the vector space $X\tensor_{\kk} A$ is naturally a $H\tensor_{\kk} A$-module and since $\lambda$ is a morphism of algebras, $X\tensor_\kk A$ is a left $A$-module. In particular, for all $x\in X, m\in M$ and $a\in A$:
\begin{equation*}
    a\cdot (x\tensor m)=(a_{(-1)}\cdot x)\tensor (a_{(0)}\cdot m).
\end{equation*}
The compatibility with the comultiplication of $H$ and the $H$-coaction provides the associativity isomorphism $m_{X,Y,M}$ and the counit provides the unit constraint on ${}_A\M$. The proof of the next theorem can be found in \cite{AM7}.
\begin{theorem}\label{exact Thm}
Let $H$ be a finite-dimensional Hopf algebra over $\kk$.
\begin{enumerate}
    \item If $A$ is right $H$-simple then ${}_A\M$ is exact, and
    \item any indecomposable exact ${}_H\M$-module category is equivalent to ${}_A\M$ for some right $H$-simple $H$-comodule algebra $A$ with trivial coinvariants.
\end{enumerate}
\end{theorem}
With this in mind, we make the following definition.
\begin{definition}
    An $H$-comodule algebra $A$ is \emph{AM-exact} if $A$ is right $H$-simple with trivial coinvariants. 
\end{definition}

We should also note that if ${}_A\M$ is exact and indecomposable, Theorem \ref{exact Thm} does not imply $A$ is AM-exact. For example, setting $A=A^{co H}=M_n(\kk)$, we see that ${}_A\M$ is an exact ${}_H\M$-module category because ${}_A\M$ is semisimple but $A$ is not AM-exact. Theorem \ref{exact Thm} guarantees the existence of an $AM$-exact $H$-comodule algebra $B$ such that ${}_A\M$ and ${}_B\M$ are equivalent. In this case $B=\kk$ suffices.

If $A,B$ are $H$-comodule algebras and ${}_A\M$ and ${}_B\M$ are equivalent as ${}_H\M$-module categories then we say $A$ and $B$ are \textit{$H$-equivariant Morita equivalent}, which we denote $A\simeq_H B$. It is clear this defines an equivalence relation on the set of exact finite-dimensional $H$-comodule algebras. Theorem \ref{exact Thm} implies that the indecomposable exact left ${}_H\M$-comodule algebras are in bijection with $H$-equivariant Morita equivalence classes of AM-exact $H$-comodule algebras.


To realize a $H$-equivariant Morita equivalence of two AM-exact $H$-comodule algebras $A$ and $B$, we will use $(A,B)$-bimodules, analogous to the study of Morita equivalence of $\kk$-algebras. We notice that if $P\in {}^H_A\M_B$, then $H\tensor_{\kk} P$ is an $(H\tensor_{\kk} A,H\tensor_{\kk} B)$-bimodule (where $H$ is an $H$-bimodule via multiplication). Thus, the coactions $\lambda_A$ and $\lambda_B$ endow $H\tensor_{\kk} P$ with the structure of an $(A,B)$-bimodule.
\begin{definition}
    We say an $(A,B)$-bimodule $M$ is an $H$-equivariant $(A,B)$-bimodule if $M$ is a left $H$-comodule such that the coaction
    \begin{equation*}
        \lambda: P\to H\tensor P
    \end{equation*}
    is a morphism of $(A,B)$-bimodules where $H\tensor P$ is a $(A,B)$-bimodule as above.
\end{definition}
We denote the category of $H$-equivariant $(A,B)$-bimodules by ${}^H_A\M_B$. The morphisms are linear maps that preserve the actions and coaction. In a similar manner we can define categories such as ${}^H\M_B$ and ${}^H_A\M$ by taking $A=\kk$ (and $B=\kk$) respectively. The following result gives an algebraic condition equivalent to equivariant Morita equivalence.
\begin{theorem}[\cite{MPT}, Proposition 3.4]\label{End P Thm}
    The algebras $A,B$ are $H$-equivariant Morita equivalent if and only if there exists $P\in {}^H\M_B$ such that $A\cong \End_B(P)$ as $H$-comodule algebras. In particular, the left $H$-comodule structure is determined by:
    \begin{equation}\label{EQU End Coact}
        \alpha(T_{(-1)})T_0(p) =\alpha\paren{T(p_{(0)})_{(-1)}S^{-1}(p_{(-1)})} T(p_{(0)})_{(0)}
    \end{equation}
    for all $\alpha\in H^*, T\in \End_B(P)$ and $p\in P$. 
\end{theorem}

Suppose $H$ is a Hopf subalgebra of $H'$. Then any left $H$-comodule algebra can be thought of as a left $H'$-comodule algebra. 
In this way, we see that an $H$-comodule algebra is AM-exact if and only if it is AM-exact as an $H'$-comodule algebra, and if $A\simeq_H B$ then $A\simeq_{H'} B$ by Theorem \ref{End P Thm}. On the other hand, in Section \ref{HKP Morita Sub}, we will show that there exists $\kK$-comodule algebras $A,B$ such that $A\not\simeq_{\kK} B$ but $A\simeq_{\hkp} B$. We should also point out an equivalence of module categories is necessarily an equivalence of abelian categories. Thus if $A\simeq_H B$, then $A$ and $B$ are Morita equivalent as $\kk$-algebras.

Viewing $H$ as a left $H$-comodule algebra via $\Delta$, a subcomodule subalgebra $K\subset H$ is called a \emph{left coideal subalgebra} of $H$. These provide an important class of AM-exact $H$-comodule algebras. In particular, every left coideal subalgebra $K\subset H$ is AM-exact. The AM-exactness of coideal subalgebras follows from the proceeding theorem which is implied by Theorem 6.1 in \cite{SKR} (which is much stronger).
\begin{theorem}\label{4.2iii_SKR}
    Let $H$ be finite-dimensional Hopf algebra and $A\subset H$ a left coideal subalgebra. Then $A\in {}^H\M_A$ is simple and every $M\in {}^H\M_A$ is free as an $A$-module.
\end{theorem}
If $A$ injects as an $H$-comodule into $H$ then $A$ is a left coideal subalgebra. Notice then if $H$ is a Hopf subalgebra of $H'$ then every coideal of $H$ is (isomorphic to) a left coideal of $H'$.

\subsection{Graded Comodule Algebras}

We say a Hopf algebra $H$ has the \textit{dual Chevalley property} if its coradical $H_0\subset H$ is a Hopf subalgebra. In this case, the coradical filtration is a Hopf algebra filtration and to this filtration, we can construct the associated graded Hopf algebra $\gr_c H:=\bigoplus_{n\geq 0} H_n/ H_{n-1}$ (with $H_{-1}=0$).

If $H$ is a coradically graded Hopf algebra with coradical $H_0$, then there exists a graded, braided Hopf algebra $R\in \yd{H_0}$ with $R(0)=\kk$ such that $H\cong R\# H_0$. We let $h\hact r$ denote the $H_0$-action and $\delta_R:R\to H_0\tensor R$ denote the $H_0$-coaction. By construction, $R$ is an $H_0$-module algebra. That is, for all $r,s\in R$ and $h\in H$:
\begin{equation*}
    h\hact (rs)=(h_{(1)}\hact r)(h_{(2)}\hact s)\text{ and } h\hact 1=\epsilon(h)1.
\end{equation*}
We refer to $R$ as the \textit{diagram} of $H$. The grading on $R$ coincides with the coradical gradation on $H$, that is $H(n)=R(n)\# H_0$ for all $n$. In particular $H(0)=1\# H_0\cong H_0$. We notice that every $M\in {}^H\M_R$ is free as an $R$-module by Theorem \ref{4.2iii_SKR}. From this point forward, we will assume $R=\B(V)=\bigoplus_{n\geq 0} \B^n(V)$ is a finite-dimensional Nichols algebra but the results hold for arbitrary (finite-dimensional) $R$. We direct the reader to \cite{ANP} for details on these constructions.

We should note here a construction of $H$-comodule algebras from an $H_0$-comodule algebra. Let $A_0$ be a left $H_0$-comodule algebra and consider the coideal subalgebras $\B(V)=\B(V)\# 1\subset H$. 
\begin{definition}
    The vector space $\B(V)\tensor A_0$ is an $H$-comodule algebra using the diagonal coproduct and multiplication
    \begin{equation*}
        (t\tensor a)(s\tensor b)=t(a_{(-1)}\cdot s)\tensor a_{(0)}b
    \end{equation*}
    for all $t,s\in \B(V)$ and $a,b\in A_0$. We denote this $H$-comodule algebra $\B(V)\# A_0$ and denote a simple tensor $t\tensor a$ by $t\# a$.
\end{definition}
This is a specific case of a more general construction of comodule algebras found in \cite{MPT}. When $A_0$ is a coideal subalgebra of $H_0$, then no confusion arises when considering $\B(V)\# A_0$ as a comodule algebra (in this manner) or as a coideal subalgebra of $H$ because these structures coincide. We note that $\B(V)\# A_0$ is generated by the subalgebras $\B(V)\# 1$ and $1\# A_0$ which are isomorphic to $\B(V)$ and $A_0$ (resp.) as $H$-comodule algebras.

Given a coradically filtered Hopf algebra $H$, and finite-dimensional $H$-comodule algebra $A$, the filtration defined by $A_k=\lambda^{-1}(H_k\tensor A)$ is an exhaustive algebra filtration called the Loewy series of $A$. The associated graded vector space $\gr_\ell A:=\bigoplus_{n\geq 0} A(n)$ (where $A(n)=A_{n}/ A_{n-1}$ with $A_{-1}=0$) is a graded algebra and the coaction $\lambda$ defines a coaction $\bar{\lambda}:\gr_\ell A\to \gr_c H\tensor \gr_\ell A$ such that $(\gr_\ell A,\bar{\lambda})$ is a Loewy-graded left $\gr_c H$-comodule algebra. That is, the Loewy series is given by $A_n=\bigoplus_{i=0}^n A(i)$, and
\begin{equation*}
    \lambda(A(n))\subset\bigoplus_{i=0}^n H(i)\tensor A(n-i)
\end{equation*}
for all $n$. For more details, see \cite{Moss}. Several properties of these filtered and graded objects are determined by $A_0$ and $H_0$. In particular, we have the following theorem.

\begin{theorem}\label{Am Lift Thm}
    Let $H_0\subset H_1\subset \dots H_m\subset H$ be a Hopf algebra filtration such that $H_0$ is a semisimple Hopf subalgebra. If $A$ is a left $H$-comodule algebra, with filtration $A_k:=\lambda^{-1}(H_k\tensor A)$, then the following are equivalent.
    \begin{enumerate}
        \item $A_0$ is an AM-exact $H_0$-comodule algebra,
        \item $\gr(A)$ is an AM-exact $\gr(H)$-comodule algebra, and
        \item $A$ is an AM-exact $H$-comodule algebra.
    \end{enumerate}
\end{theorem}
\begin{proof}
    The proof follows directly from \cite{MPT} Proposition 4.4 (and Corollary 4.5). From there it is clear $A^{\co H}=A_0^{\co H_0}$ which proves the result.
\end{proof}
In particular if $H$ has the dual Chevalley property, then the coradical filtration is a Hopf algebra filtration and $H_0$ is semisimple. These observations motivate the following three step process for the classification of isomorphism classes of $H$-comodule algebras from \cite{MPT}. First, determine all AM-exact $H_0$-comodule algebras $A_0$. Then determine all graded comodule algebras $\gr_\ell A:=\bigoplus_{n\geq 0} A(n)$ such that $A_0=A(0)$. Lastly determine all lifts of $\gr_\ell A$ along lifts of $\gr_c H$. A proper discussion can be found in \cite{MPT} along with some explicit computations of $H$-equivariant Morita equivalence class representatives for several pointed Hopf algebras via this method. The final step of the proposed method appears to be one of significant difficulty with few tools to utilize systematically. However, we should notice that it may not be necessary to determine the isomorphism class of all AM-exact $H$-comodule algebras to determine representatives for each $H$-equivariant Morita equivalence class.

%% file: Sections/3.LG_Endo.tex
\section{Endomorphisms of Loewy-graded modules}

\input{Sections/3.LG_Endo_sub/General}

\subsection{A Meaningful Example}

\input{Sections/3.LG_Endo_sub/Example}

%% file: Sections/3.LG_Endo_sub/General.tex
Let $H_0$ be a semisimple Hopf algebra, $\B(V)\in\yd{H_0}$ a finite-dimensional Nichols algebra and set $H$ to be the coradically graded Hopf algebra given by the bosonization $H=\B(V)\# H_0$. Let $A$ be a Loewy-graded AM-exact $H$-comodule algebra. Theorem \ref{Am Lift Thm} implies $A_0$ is an AM-exact $H_0$-comodule algebra. Since $H(0)$ is semisimple, ${}_{H(0)}\M$ is fusion, which implies $A_0$ is a semisimple $\kk$-algebra.

\subsection{Graded Actions of Loewy-Graded $H$-Comodule Algebras}
Let $B=\oplus_{n\geq 0} B(n)$ be a Loewy-graded $H$-comodule algebra. We should note that $B_0=B(0)$ and we will use these interchangeably.
\begin{definition}\label{Graded Equiv Ob Def}
    We say that $P\in {}^H\M_B$ is a \emph{Loewy-graded $B$-module} if:
    \begin{enumerate}
        \item $P$ decomposes as a $B(0)$-module, $P=\bigoplus_{n\geq 0} P(n)$, such that $P(n)\cdot B(k)\subset P(n+k)$ for all $n,k$, and
        \item The Loewy-series for $P$ is given by $P_n=\bigoplus_{k=0}^n P(k)$.
    \end{enumerate}
    Similarly, we can define Loewy-graded modules in ${}^H_A\M$ and we say $P\in {}^H_A\M_B$ is a \emph{Loewy-graded $H$-equivariant $(A,B)$-bimodule} if it is graded in both ${}^H_A\M$ and ${}^H\M_B$.
\end{definition} 


We now discuss some properties of $\End_B(P)$ for a Loewy-graded $B$-module $P\in {}^H\M_B$. The first thing to notice is that $\B(V)$ injects (as an $H$-comodule algebra) into $B$. Thus every $P\in {}^H\M_B$ forgets to an object in ${}^H\M_{\B(V)}$ which implies $P$ is a free $\B(V)$-module by Theorem \ref{4.2iii_SKR}.

Assume $P$ is generated as a $B$-module by $P(0)$ and set $S:=\End_B(P)$. By freeness, $P\cong P(0)\tensor \B(V)$ as $\B(V)$-modules. Thus any basis for $P(0)$ can be taken as a free $\B(V)$-basis for $P$ and any $B_0$-linear morphism $\phi:P(0)\to P$ extends to a $\B(V)$-module morphism. Since the morphism commutes with the action of $B_0$ and $\B(V)$, this map defines an element of $S$. Similarly, given any $\Phi\in S$, the restriction $\Phi|_{P(0)}$ is an element in $\Hom_{B(0)}(P(0),P)$. It is clear that these maps realize a linear isomorphism:
\begin{equation}\label{eq:ScongX}
S\cong  X:=\Hom_{B(0)}(P(0),P).
\end{equation}
Since $P$ is a Loewy-graded $B$-module and $B_0$ is semisimple, the projections $\pi_{P(n)}:P\to P(n)$ are $B_0$-linear. Thus $X=\bigoplus_{n\geq 0} X(n)$ where $X(n)=\Hom_{B_0}(P(0),P(n))$. 

By semisimplicity, there exists a finite set $I$ and subset $\set{p_i}_{i\in I}\subset P(0)$ such that $P(0)=\bigoplus_{i\in I} p_iB_0$ with each $p_iB_0$ a simple $B_0$-module. Thus $X(n)=\Hom_{B_0}(P(0), P(n))=\bigoplus_{i\in I}\Hom_{B_0}(p_iB_0, P(n))$ and so every morphism in $X(n)$ can be decomposed as the sum of $B_0$-linear maps that vanish on all but one of $\set{p_i}_{i\in I}$ with image in $P(n)$. Moreover, since $\set{p_i}_{i\in I}$ generates $P(0)$ as a $B_0$-module, it generates $P$ as a $B$-module. That is an arbitrary element in $P(n)$ is of the form $\sum_{j\in I} p_j\cdot b_j^n$ for some elements $b_j^n\in B(n)$ for each $j\in I$.

\begin{proposition}\label{S is Grade prop}
   The linear isomorphism in \eqref{eq:ScongX} pushes the vector space grading of $X$ to $S$. With this grading, $S$ is a graded algebra.
\end{proposition}
\begin{proof}
    Consider $\phi^{i}_n\in \Hom_{B_0}(p_iB_0,P(n))$ and $\phi_j^k\in \Hom_{B_0}(p_jB_0,P(k))$. There exists $b_n^i\in B(n)$ and $b_k^i\in B(k)$ for each $i\in I$ such that:
    \begin{align}
        \phi_n^i(p_i)&=\sum_{m\in I} p_m\cdot b_n^m, && \phi_k^j(p_j)=\sum_{m\in I} p_m\cdot b_k^m.
    \end{align}
    Now we notice that
    \begin{equation*}
    \phi_k^j\circ \phi_n^i(p_i)=\phi_k^j\paren{\sum_{m\in I} p_m\cdot b_n^m}=\phi_k^j(p_j\cdot b_n^j)=\phi_k^j(p_j)b_n^j =\sum_{m\in I} p_m\cdot (b_k^m b_n^j)\in P(0)\cdot B(n+k)\subset P(n+k).
    \end{equation*}
    Thus $\phi_k^j\circ \phi_n^i\in \Hom_{B_0}(p_iB_0,P(n+k))$. Pushing along the isomorphism to $S$, we see that this implies $S(n)S(k)\subset S(n+k)$. Since $n,k$ were arbitrary, the result holds.
\end{proof}

\begin{theorem}\label{Loew Grade S Thm}
    Let $P\in{}^H\M_B$ be a Loewy-graded $B$-module, generated by $P(0)$. Then $S=\End_B(P)$ is a Loewy-graded $H$-comodule algebra.
\end{theorem}
\begin{proof}
    We have already shown that $S$ has the structure of a graded algebra. We will now show that the filtration, $S^n:=\bigoplus_{i=0}^n S(i)$, coincides with the Loewy filtration, $S_n$.

    Let $\phi\in S_n$. Then for all $i$, 
    \begin{equation*}
        \lambda_P(\phi(p_i))=\lambda_S(\phi)\lambda_P(p_i)\in (H_n\tensor S)(H_0\tensor P(0))\subset H_n\tensor P.
    \end{equation*} 
    This implies $\phi(p_i)\subset \bigoplus_{k=0}^n P(k)$ for all $i$. Thus $\phi$ restricts to a map in $\bigoplus_{k=0}^n \Hom_{B_0}(P(0), P(k))=\bigoplus_{k=0}^n X(k)$. Pushing along the isomorphism to $S$, we see $\phi\in S^n$ and since $\phi$ was arbitrary, $S_n\subset S^n$.

    Since an element of $S$ is entirely determined by the image of $P(0)$, we see that $\phi\in S_n$ if and only if 
    \begin{equation}\label{eq alpha}
        \alpha(\phi_{(-1)})\phi_{(0)}(p)=0
    \end{equation}   
    for all $\alpha\in H^*$ with $H_n\subset \ker{\alpha}$ and all $p\in P(0)$. By construction (see \eqref{EQU End Coact})
    \begin{equation*}
        \alpha(\phi_{(-1)})\phi_{(0)}(p)=\alpha\paren{\phi(p_{(0)})_{(-1)}S^{-1}(p_{(-1)})}\phi(p_{(0)})_{(0)}.
    \end{equation*}
    
    If $\phi\in S^n$ then $\phi(P(0))\subset \bigoplus_{k=0}^n P(k)$. Since $p_{(-1)}\tensor p_{(0)}\subset H_0\tensor P(0)$ for all $p\in P(0)$, we have 
    \begin{equation*}
        p_{(-1)}\tensor \phi(p_{(0)})\in H_0\tensor \paren{\oplus_{k=0}^n P(k)}=H_0\tensor P_n.
    \end{equation*}
    Thus, for all $p\in P(0), \phi(p_{(0)})_{(-1)}\tensor \phi(p_{(0)})_{(0)}\in H_n\tensor P$. Since $H_0$ is a Hopf subalgebra of $H$, $S^{-1}(p_{(-1)})\in H_0$ and thus
    \begin{equation*}
        \paren{\phi(p_{(0)})_{(-1)}S^{-1}(p_{(-1)})}\tensor \phi(p_{(0)})_{(0)}\in H_n\tensor P
    \end{equation*}
    thus (\ref{eq alpha}) holds for all $\alpha$ vanishing on $H_n$ and $p\in P(0)$. This implies $\phi\in S_n$ and since $\phi$ was arbitrary, $S^n\subset S_n$.
\end{proof}
\begin{corollary}\label{Deg 0 Corollary}
    If $P\in {}^H\M_B$ is a Loewy-graded $B$-module that is generated in degree zero and $S=\End_B(P)$, then $S(0)$ is isomorphic to $\End_{B_0}(P(0))$ and $S(0)\simeq_{H_0} B_0$.
\end{corollary}
\begin{proof}
    The graded linear isomorphism \eqref{eq:ScongX} descends to an algebra isomorphism $S(0)\cong X(0)$. The $H_0$-equivariant Morita equivalence of $S(0)$ and $B_0$ then follows from Theorem \ref{End P Thm}.
\end{proof}

%% file: Sections/3.LG_Endo_sub/Example.tex
Again, $H=\B(V)\# H_0$ is a coradically graded Hopf algebra and $A,B$ Loewy-graded $H$-comodule algebras. In general it is not the case that $A(0)\simeq_{H_0}B(0)$ implies $A\simeq_H B$. For example, taking $A=H$ then we see $A(0)$ and $H(0)$ are isomorphic so $A(0)\simeq_{H_0} H(0)$. However, $H\not\simeq_H H_0$ because $H_0$ is semisimple and $H$ is not, thus ${}_H\M$ and ${}_{H_0}\M$ are not equivalent as abelian categories, much less equivalent as ${}_H\M$-module categories. Instead we ask whether a Loewy-graded $H$-comodule algebra $A$ and $H_0$-equivariant Morita equivalence in degree zero is induced from an $H$-equivaraint Morita equivalence.
\begin{question}\label{Q Morita Ext}
    Let $A$ be an AM-exact Loewy-graded left $H$-comodule algebra. For every $H_0$-equivariant Morita equivalence $A(0)\simeq_{H_0} X$, does there exists a Loewy-graded left $H$-comodule algebra $B$ (isomorphic to $X$ in degree zero) such that $A\simeq_H B$?
\end{question}
We will now answer this question affirmatively in the case when $A=\B(V)\# A(0)$. In particular we will construct a graded $A$-module $P\in {}^H\M_A$ satisfying the hypotheses of Theorem \ref{Loew Grade S Thm} and show $B=\End_A(P)$ satisfies the conclusion.

\begin{proposition}\label{Extend P0 Prop}
    Let $A_0$ be a left $H_0$-comodule algebra and set $A=\B(V)\# A_0$. Suppose $P_0\in {}^{H_0}\M_{A_0}$, then $P:=P_0\tensor \B(V)\in {}^H\M_A$ where the right $A$-action is given by
    \begin{equation*}
        (p\tensor s)\cdot (t\# a)=(p\cdot a_{(0)})\tensor \bracks{S^{-1}(a_{(-1)})\hact (st)} 
    \end{equation*}
    for all $p\in P_0, s,t\in \B(V)$ and $a\in A_0$ where $\hact$ is the right $H_0$ action on $\B(V)\in\yd{H_0}$. The coaction $\lambda_P:P\to H\tensor P$ is given by the diagonal coaction (where we view $\B(V)$ as a coideal subalgebra of $H$).
\end{proposition}
\begin{proof}
    Let $P_0, A_0$ be as given. To see that $P\in \M_A$, we first see that $1\# 1$ acts by the identity. It is also clear that $P\in \M_{\B(V)}$. Moreover, since $\B(V)$ is a right $H_0$-module via $S^{-1}$, we see that the action defines a $A_0$-module structure. Thus we see that for all $p\in P_0, r,t\in \B(V)$, and $a\in A$:
    \begin{align*}
        (p\tensor r)\cdot ((t\# 1)(1\# a))&=(p\tensor r)\cdot (t\# a)\\
        &=p\cdot a_{(0)}\tensor \bracks{S^{-1}(a_{(-1)})\hact (rt)}\\
        &=(p\tensor rt)\cdot (1\# a)\\
        &=((p\tensor r)\cdot (t\# 1))\cdot(1\# a)
    \end{align*}
    and
    \begin{align*}
        (p\tensor r)\cdot ((1\# a)(t\# 1))&=(p\tensor r)\cdot ((a_{(-1)}\cdot t)\# a_{(0)})\\
        &=(p\cdot a_{(0)})\tensor \bracks{S^{-1}(a_{(-1)})\hact (r(a_{(-2)}\hact t))}\\
        &=(p\cdot a_{(0)})\tensor \bracks{S^{-1}(a_{(-1)})\hact r}\bracks{S^{-1}(a_{(-2)})\hact (a_{(-3)}\hact t)}\\
        &=(p\cdot a_{(0)})\tensor \bracks{S^{-1}(a_{(-1)})\hact r}\bracks{S^{-1}(a_{(-2)})a_{(-3)}\hact t}\\
        &=(p\cdot a_{(0)})\tensor \bracks{S^{-1}(a_{(-1)})\hact r}\bracks{\epsilon(a_{(-2)})\hact t}\\
        &=(p\cdot a_{(0)})\tensor \bracks{S^{-1}(a_{(-1)}\epsilon(a_{(-2)}))\hact r}t\\
        &=(p\cdot a_{(0)})\tensor (S^{-1}(a_{(-1)})\hact r)t\\
        &=\bracks{(p\cdot a_{(0)})\tensor S^{-1}(a_{(-1)})\hact r)}\cdot (t\# 1)\\
        &=\bracks{(p\tensor r)\cdot (1\# a)}\cdot (t\# 1).
    \end{align*}
    Since $A$ is generated by $\B(V)\# 1$ and $1\# A_0$, the computations above imply $P\in \M_A$. We note that $A$ is free as a $\B(V)$-module of rank $\dim(P_0)$.
 
    To show $P\in {}^H\M_A$, we need to show $\lambda_P$ is a morphism of right $A$-modules. Notice that $P$ is generated as a left $A$ module by $P_0=P_0\tensor 1_{\B(V)}\subset P$. Thus since $\lambda_A$ is a morphism of algebras, it is sufficient to show
    \begin{equation*}
    \lambda_P((p\tensor 1)\cdot (t\# a))=\lambda_P(p\tensor 1)\cdot \lambda_B(t\# a))
    \end{equation*}
    for all $p\in P_0, t\in \B(V)$ and $a\in A_0$ and since $A$ is generated by $\B(V)$ and $A_0$, it is sufficient to prove the two cases: $t=1$ or $a=1$. First, if $t=1$, then
    \begin{align*}
        \lambda_P((p\tensor 1)\cdot (1\# a))&=\lambda_P\paren{(p\cdot a_{(0)})\tensor \epsilon(S^{-1}(a_{(-1)}))}\\
        &=\lambda_P(p\cdot a_{(0)}\epsilon(a_{(-1)})\tensor 1)\\
        &=\lambda_P((p\cdot a)\tensor 1)\\
        &=(p_{(-1)}a_{(-1)})\tensor ((p_{(0)}\cdot a_{(0)})\tensor 1)\\
        &=(p_{(-1)}a_{(-1)})\tensor ((p_{(0)}\cdot (a_{(0)})_{(0)})\tensor \epsilon(S^{-1}((a_{(0)})_{(-1)}))\\
        &=(p_{(-1)}\tensor (p_{(0)}\tensor 1))\cdot (a_{(-1)}\tensor (1\# a_{(0)}))\\
        &=\lambda_P(p\tensor 1)\cdot \lambda_B(1\# a).
    \end{align*}
    Similarly, if $a=1$, then
    \begin{align*}
        \lambda_P((p\tensor 1)\cdot (t\# 1))&=\lambda_P(p\tensor t)\\
        &=p_{(-1)}t_{(-1)}\tensor (p_{(0)}\tensor t_{(0)})\\
        &=(p_{(-1)}\tensor (p_{(0)}\tensor 1))\cdot (t_{(-1)}\tensor (t_{(0)}\# 1))\\
        &=\lambda_P(p\tensor 1)\cdot \lambda_B(t\# 1).
    \end{align*}
    These computations show $P\in {}^H\M_A$.
\end{proof}
From here we can notice that $P$ inherits a grading from $\B(V)$. In particular, we can see that $P=\bigoplus_{n=0} P(n)$ as vector spaces where $P(n):=P_0\tensor \B^n(V)$. We now show that $P$ is a Loewy-graded $A$-module with this grading.

\begin{proposition}\label{Examp is Graded P prop}
    $P$ is a Loewy-graded $A$-module with the grading $P(n)=P_0\tensor \B^n(V)$. Moreover $P$ is generated in degree zero as a $A$-module.
\end{proposition}
\begin{proof}
    Recall $A=\B(V)\#A_0$ is Loewy-graded with $A(n)=\B^n(V)\# A_0$. We note once more that any basis for $P(0)$ is a free basis for $P$ as a $\B(V)$-module. This implies $P$ is generated in degree zero as an $A$-module. 
    
    Now we notice the action of $A_0$ on $P$ preserves the grading. That is for all $n$,
    \begin{equation*}
        P(n)\cdot A(0)=(P_0\tensor \B^n(V))\cdot (1\# A_0)\subset (P_0\cdot A_0)\tensor (H_0\hact \B^n(V))\subset P_0\tensor \B^n(V)=P(n).
    \end{equation*}
    Condition 1 of Definition \ref{Graded Equiv Ob Def} holds by the computation above and the fact that $\B(V)$ is a graded algebra. Condition 2 follows from the fact that  $\B(V)$ is a Loewy-graded $H$-comodule algebra and because the coaction takes $P_0$ into $H_0\tensor P_0$.
\end{proof}
This partially answers Question \ref{Q Morita Ext} in the following sense:
\begin{corollary}\label{Crossed Morita Ext Cor}
    If $A_0\simeq_{H_0} X$, then $\B(V)\# A_0$ is $H$-equivariant Morita equivalent to a Loewy-graded left $H$-comodule algebra $B$ with $B(0)\cong X$.
\end{corollary}
\begin{proof}
    Let $P_0\in {}^H\M_{A_0}$ such that $X\cong\End_{A_0}(P_0)$. The construction in Proposition \ref{Extend P0 Prop} realizes an $H$-equivariant equivalence between $\B(V)\# A_0$ with $B:=\End_{\B(V)\# A_0}(P)$. Proposition \ref{Examp is Graded P prop} and Theorem \ref{Loew Grade S Thm} show $B$ is a Loewy-graded $H$-comodule algebra and Corollary \ref{Deg 0 Corollary} implies $B(0)= \End_{B_0}(P_0)\cong X$.
\end{proof}

%% file: Sections/4.LG_AMEx_HCMA.tex
\section{Loewy-graded AM-exact $H$-comodule algebras}

In this section we prove Theorem \ref{Main Thm s1} and discuss the structure of Loewy-graded comodule algebras. Let $H_0$ be a semisimple Hopf algebra and $\B(V)\in \yd{H_0}$ a finite-dimensional Nichols algebra. Let $H:=\B(V)\# H_0$ and $A$ a Loewy-graded AM-exact $H$-comodule algebra.

\begin{lemma}\label{Kappa Lemma}
    The map $\kappa:=(\id_H\tensor \pi_{A(0)})\lambda:A\to H\tensor A(0)$ is an injective morphism of $H$-comodule algebras.
\end{lemma}
\begin{proof}
    Lemma 5.1 in \cite{SKR} proves $\kappa$ is an $H$-comodule algebra morphism. This lemma also tells us $\ker(\kappa)$ is an $H$-costable (two sided) ideal of $A$. Since $A$ is AM-exact, $A$ contains no (non-trivial) $H$-costable right ideals so $\ker(\kappa)=0$.
\end{proof}
We also notice that the morphism $\kappa$ is graded. That is $\kappa(A(n))\subset (H\tensor A(0))(n):=H(n)\tensor A(0)$. If $A(0)$ is isomorphic to a subcomodule of $H_0$, then the counit $\epsilon=\epsilon_{H_0}$ defines a linear map $\epsilon_{A(0)}:A(0)\to \kk$ which we extend to $A$ by precomposition with the projection $\pi_{A(0)}:A\to A(0)$.

\begin{proposition}\label{subcomod prop}
    If $A$ is a Loewy-graded AM-exact $H$-comodule algebra and $A(0)$ is isomorphic to a subcomodule of $H(0)$ then $A$ is isomorphic to a subcomodule of $H$.
\end{proposition}
\begin{proof}
    Consider the map $\phi:A\to H$ defined by the composition $\phi:=(\id_H\tensor\epsilon_{A(0)}\pi_{(0)})\lambda$. Notice the squares in the diagram below commute:
\[\begin{tikzcd}
	A && {H\tensor A} && {H\tensor \kk\cong H} \\
	{H\tensor A} && {H\tensor H\tensor A} && {H\tensor H\tensor \kk\cong H\tensor H}
	\arrow["\lambda", from=1-1, to=1-3]
	\arrow["\lambda"', from=1-1, to=2-1]
	\arrow["{\id_H\tensor \epsilon_{A(0)}\pi_{A(0)}}", from=1-3, to=1-5]
	\arrow["{\Delta\tensor \id_A}"', from=1-3, to=2-3]
	\arrow["{\Delta\tensor \id_\kk}", from=1-5, to=2-5]
	\arrow["{\id_H\tensor \lambda}"', from=2-1, to=2-3]
	\arrow["{\id_H\tensor \epsilon_{A(0)}\pi_{A(0)}}"', from=2-3, to=2-5]
\end{tikzcd}\]
    which implies $\Delta\circ \phi=(\id_H\tensor \phi)\lambda$.
    
    Now we will show $\phi$ is injective. By assumption $\phi$ is injective on $A(0)$. We also note $H\cong \bigoplus_{i=1}^n H(i)$ as vector spaces and so we let $\pi_i:H\to H(i)$ to be the linear projection for each degree $i$. Let $a\in\ker(\phi)$ and let $n>0$ be minimal such that $a\in A_n$. Then $a=\sum_{i=0}^n a_i$ with $a_n\neq 0$ and $a_i\in A(i)$ for all $i$. Notice then $\Delta\phi(a)=(1\tensor \phi)\lambda(a)=0$. In particular, the projection of $(1\tensor \phi)\lambda(a)$ on $H(n)\tensor H(0)$ must vanish. It is clear that the projection of $(1\tensor \phi)\lambda(a)$ on $H(n)\tensor H(0)$ is equal to the projection of $(1\tensor \phi)\lambda(a_n)$ on $H(n)\tensor H(0)$. That is if $\phi(a)$ vanishes and $\lambda(a_n)=\sum_{i=1}^n b_{i,n-i}$ where $b_{i,n-i}\in H(i)\tensor A(n-i)$, then
    \begin{equation*}
        (\pi_n\tensor \pi_0)(1\tensor \phi)\lambda(a_n)=0.
    \end{equation*}
    We then see:
    \begin{align*}
        (\pi_n\tensor \pi_0)(1\tensor \phi)\lambda(a_n)&=(1\tensor \pi_0)(1\tensor\phi)(\pi_n\tensor 1)\lambda(a_n)\\
        &=(1\tensor \pi_0)(1\tensor \phi)(b_{n,0}).
    \end{align*}
    Now since $b_{n,0}\in H(n)\tensor A(0)$,  $(1\tensor \phi)(b_{n,0})\in H(n)\tensor H(0)$ which implies
    \begin{equation*}
        0=(\pi_n\tensor \pi_0)(1\tensor \phi)\lambda(a_n)=(1\tensor \phi)(b_{n,0}).
    \end{equation*}
    Since $\phi$ is injective on $A(0)$, this implies $b_{n,0}=0$. This contradicts the minimality of $n$ above. Thus $\ker(\phi)=0$.
\end{proof}
\begin{corollary}\label{CSA Cor}
    If $A$ is a Loewy-graded AM-exact $H$-comodule algebra and $A(0)$ is isomorphic to a coideal subalgebra of $H(0)$ then $A$ is isomorphic to a coideal subalgebra of $A$
\end{corollary}
\begin{proof}
    The isomorphism $\phi$, as in Proposition \ref{subcomod prop}, is the composition of $\kk$-algebra morphisms.
\end{proof}
Proposition \ref{subcomod prop} is a minor generalization of Corollary \ref{CSA Cor}, which we note was proven in \cite{MBP} (Proposition 5.3). Consequently, the above proof is taken almost directly from these results.

Let $A, B$ be Loewy-graded $H$-comodule algebras with $\phi:A(0)\to B(0)$ an isomorphism of $H_0$-comodule algebras. Then the map $1\tensor \phi:H\tensor A(0)\to H\tensor B(0)$ is an isomorphism of $H$-comodule algebras (each equipped with coaction $\Delta\tensor 1$) such that $(1\tensor \phi)\kappa_A=\kappa_B\phi$ on $A(0)$ (where $\kappa_A, \kappa_B$ are isomorphisms from Lemma \ref{Kappa Lemma}). In particular, if $A(0)\cong B(0)$, then we can view $A,B$ as $H$-costable subalgebras of $H\tensor B(0)$ containing $\kappa_B(B(0))=\lambda_B(B(0))$. In particular, $A\cong \kappa_A(A)$ and $B\cong \kappa_B(B)$ so we define $A\cap B$ to be the isomorphism class of $\kappa(A)\cap \kappa(B)$ as a subspace of $H\tensor B(0)$. We notice that the intersection of $H$-subcomodules is a subcomodule and the intersection of subalgebras is a subalgebra. Thus it is clear $A\cap B$ is a Loewy-graded $H$-comodule algebra with $(A\cap B)(0)\cong A(0)\cong B(0)$. Further $A\cap B$ is AM-exact by Theorem \ref{Am Lift Thm}.

From here we notice that there is a largest such AM-exact comodule algebra.
To determine this algebra, we must establish some notation. Let $d=\dim(A(0))$. Since the coaction is defined by $\Delta\tensor 1$ on $H\tensor A(0)$, we have an $H$-comodule isomorphism
\begin{equation*}
    (H\tensor A(0), \Delta\tensor 1)\cong (H^d,\Delta^d) 
\end{equation*}
where $\Delta^d$ is the map restricting to $\Delta$ on each copy of $H$. Similarly, define $\epsilon^d$ as the map restricting to $\epsilon$ on each copy of $H$. We will view $A$ as a subcomodule of $H^d$ via the isomorphism $A\cong \kappa_A(A)\subset H\tensor A(0)\cong H^d$ from Lemma \ref{Kappa Lemma}.

\begin{proposition}\label{Sub Crossed Prop}
    If $A$ is a Loewy-graded AM-exact $H$-comodule algebra then $A$ injects into $\B(V)\# A(0)$.
\end{proposition}
\begin{proof}
    We have an isomorphism $H\cong \B(V)\tensor H_0$ of $H$-comodules, where $\B(V),H_0$ are $H$-comodules via $\Delta$. Thus $H^d\cong \B(V)\tensor (H_0)^d$ as comodules. Moreover, $A(0)$ must be contained in $1\tensor (H_0)^d$. By cosemisimplicity, there exists an $H_0$-comodule compliment $D$ such that $A(0)\oplus D\cong (H_0)^d$. This extends to a decomposition of $H$-comodules:
    \begin{equation*}
        H^d = (\B(V)\tensor A(0))\oplus (\B(V)\tensor D).
    \end{equation*}
    Consider the linear map $\alpha:H^d\to \kk$ vanishing on $\B(V)\tensor A(0)$ and restricting to $\epsilon^d$ on $\B(V)\tensor D$. We notice that for $n\geq 1$,  $\epsilon$ vanishes on $H(n)$. Thus $\alpha(A(n))=0$ for $n>0$. By construction, $\alpha (A(0))=0$. Thus $\alpha(A)=0$. The map  $(1\tensor \alpha)\Delta^d$ restricts to the identity on $\B(V)\tensor D$ but
    \begin{equation*}
        (1\tensor \alpha)\Delta^d(A)\subset (1\tensor \alpha)(H\tensor A)=0.
    \end{equation*}
    Thus $A\cap (\B(V)\tensor D)=0$ and so,  
    \begin{equation*}
        A\hookrightarrow \B(V)\tensor A(0)\subset H^d
    \end{equation*}
    as an $H$-comodule. Thus there is an injection of $H$-comodules $A\hookrightarrow \B(V)\# A(0)$. Since the multiplication on $A$ and $\B(V)\# A(0)$ agree in degree zero, they agree on their intersection by Lemma \ref{Kappa Lemma}. Thus $A$ is isomorphic to an $H$-costable subalgebra of $\B(V)\# A(0)$.
\end{proof}

We now prove the main result of this paper, which implies Theorem \ref{Main Thm s1}. We also note two immediate corollaries.
\begin{theorem}\label{Main Thm s1 Sec 4}
    If $A$ is a Loewy-graded, AM-exact $H$-comodule algebra and $A(0)\simeq_{H_0} X$ is $H_0$-equivariant Morita equivalence, then there exists a Loewy-graded, AM-exact $H$-comodule algebra $B$ such that $A\simeq_H B$ and $B(0)\cong X$ as $H_0$-comodule algebras.
\end{theorem}
\begin{proof}
    First we notice by Proposition \ref{Sub Crossed Prop}, $A\subset \B(V)\# A(0)$. Corollary \ref{Crossed Morita Ext Cor} implies $\B(V)\# A(0)$ is $H$-equivariant Morita equivalent to a Loewy-graded $H$-comodule algebra $B$ with $B(0)\cong X$. In particular, there exists a Loewy-graded $\B(V)\#A(0)$-module $P$ such that 
    $\B(V)\# A(0)\simeq_H B=\End_{\B(V)\# A(0)}(P)$ and $B(0)\cong \End_{A(0)}(P(0))\cong X$.
    
    We notice $P$ forgets to an object in ${}^H\M_A$. Take $P_A$ to be the $A$-submodule of $P$ generated by $P(0)$. We notice $P$ is $H$-costable because
    \begin{equation*}
        \lambda_P(P_A)=\lambda_P(P(0)\cdot A)=\lambda_P(P(0))\lambda_A(A)\subset (H(0)\tensor P(0))(H\tensor A)\subset H\tensor P_A.
    \end{equation*}
    We set $P_A(n):=P(0)\cdot A(n)$. With this grading, $P_A$ is a Loewy-graded $A$-module. Now, by Theorem 3.3, $B_A:=\End_A(P_A)$ is a Loewy-graded $H$-comodule algebra and the equality $P_A(0)=P(0)$ (of $A(0)$-modules) realizes an isomorphism $B_A(0)\cong B(0)\cong X$ by Corollary \ref{Deg 0 Corollary}. Thus $B_A$ has the desired properties.
\end{proof}

\begin{corollary}\label{Morita Coideal Cor}
    If $A$ is an Loewy-graded AM-exact $H$-comodule algebra such that $A(0)$ is $H_0$-equivariant Morita equivalent to a coideal subalgebra of $H_0$, then $A$ is $H$-equivariant Morita equivalent to a homogeneous coideal subalgebra of $H$.
\end{corollary}
\begin{proof}
    By Theorem \ref{Main Thm s1 Sec 4}, $A$ is $H$-equivariant Morita equivalent to a Loewy-graded $H$-comodule algebra $B$ such that $B(0)$ is a coideal subalgebra. Corollary \ref{CSA Cor} implies $B$ is isomorphic to a coideal subalgebra.
\end{proof}

\begin{corollary}\label{All Coideal cor}
    If in addition, every $H_0$-equivariant Morita equivalence class is represented by a coideal subalgebra of $H_0$ then every Loewy-graded $H$-comodule algebra is $H$-equivariant Morita equivalent to a coideal subalgebra of $H$.
\end{corollary}
\begin{proof}
    This follows directly from Corollary \ref{Morita Coideal Cor}.
\end{proof}

It is straightforward to see $H_0=\GA{\Z/n\Z}$ satisfies the hypothesis of Corollary \ref{All Coideal cor}. In the next section, we show that the Kac-Paljutkin Hopf algebra, $\hkp$, also satisfies the hypotheses of Corollary \ref{All Coideal cor}.

%% file: Sections/5.KP_Mod_Cats.tex
\section{Exact ${}_\hkp\M$-module categories}

\input{Sections/5.H8_sub/Intro_H8}

\subsection{Small Dimensional Cases}
\input{Sections/5.H8_sub/Small_Dim}
\subsection{Larger Cases}

\input{Sections/5.H8_sub/Larger_Cases}

%% file: Sections/5.H8_sub/Intro_H8.tex
\begin{definition}[The Kac-Paljutkin Hopf Algebra]
    Let $\kk$ be an algebraically closed field of characteristic zero. Consider the $\kk$-algebra generated by $x,y,z$ satisfying the relations:
    \begin{align*}
        x^2=y^2=1,&& xy=yx,&& yz=zx\\
        xz&=zy,& z^2&=\frac{1}{2}(1+x+y-xy).
    \end{align*}
    This algebra admits a bialgebra structure by setting $x,y$ grouplike, and
    \begin{equation*}
        \Delta(z)=\frac{1}{2}((1+y)z\tensor z+(1-y)z\tensor xz).
    \end{equation*}
    With $\epsilon(z)=1$, $S(z)=z$, we can see this defines a Hopf algebra, which we denote $\hkp$.
\end{definition}
The Hopf algebra $\hkp$ is the unique semismimple Hopf algebra over $\kk$ of dimension $8$ that is not a group algebra, nor dual to a group algebra. Moreover, it is smallest dimensional Hopf algebra with this property (see \cite{Masu}). This fact suggests $\hkp$ is a reasonable place to begin investigating non-pointed Hopf algebras.

We take a basis for $\hkp$ given by $\mathcal{B}=\set{1,x,y,xy,z,zx,zy,zxy}$. We notice that the subalgebra generated by $x$ and $y$ forms a Hopf subalgebra isomorphic to $\kK$ where $K$ is the Klein four group. For every element $b\in \mathcal{B}$ let $b^*:\hkp\to \kk$ be the linear map vanishing on $\mathcal{B}\backslash\set{b}$ and $b^*(b)=1$.

We note that there are 5 non-isomorphic simple left $\hkp$-comodules: four one dimensional comodules $\kk_g$, with coaction $\lambda(\kk_g)=g\tensor \kk_g$ and one two dimensional simple comodule $X$ with basis $v,w$ and coaction:
\begin{align*}
    \lambda(v)&=\half \Big[(1+y)z\tensor v+(1-y)z\tensor w\Big]\\
    \lambda(w)&=\half \Big[x(1-y)z\tensor v+x(1+y)z\tensor w\Big].
\end{align*}
By self duality, we can see $\hkp$ has five (non-isomorphic) simple representations as well. In particular, there are four one dimensional representations induced by the algebra morphisms $\set{\eta_\zeta}_{\zeta^4=1}$ from $\hkp\to \kk$ defined by $\eta_\zeta(z)=\zeta$. This forces $\eta_\zeta(x)=\eta_\zeta(y)=\zeta^2$ and entirely determines distinct algebra morphisms $\hkp\to \kk$. Let $\kk_\zeta$ be the one dimensional representation associated to $\eta_\zeta$. We should notice that any algebra morphism $\eta:\hkp\to \kk$ restricts to an algebra morphism $\eta|_{\kK}:\kK\to \kk$. We should note however that not all algebra morphisms $\mu:\kK\to \kk$ arise this way. In particular, the algebra morphisms from the group algebra $\kK$ to $\kk$ are of the form $\mu_{a,b}=1^*+ax^*+bx^*+ab(xy)^*$ for $a,b\in\set{\pm 1}$ but $\eta_\zeta|_{\kK}\in\set{\mu_{1,1},\mu_{-1,-1}}$ for all $\zeta$. We should note that this implies the forgetful functor ${}_{\hkp}\M\to {}_{\kK}\M$ is not essentially surjective.

We let $A$ be an arbitrary left $\hkp$-comodule. Then by cosemisimplicity, $A\cong \paren{\bigoplus_{g\in K} \kk_g^{n_g}}\oplus X^{n_2}$ for some non-negative integers $n_1, n_x, n_y, n_{xy}, n_2$. Thus we can take a basis for $A$ given by $\set{a_g^i}_{g\in K, 1\leq i\leq n_g}\cup\set{v_i,w_i}_{1\leq i\leq n_2}$ with
\begin{align*}
    \lambda(a_g^i)&=g\tensor a_g^i\\
    \lambda(v_i)&=\half \Big[(1+y)z\tensor v_i+(1-y)z\tensor w_i\Big]\\
    \lambda(w_i)&=\half \Big[x(1-y)z\tensor v_i+x(1+y)z\tensor w_i\Big].
\end{align*}
We set $A_K$ to be the sum of the one dimensional $\hkp$-subcomodules of $A$ and set $A_2$ be the sum of the simple $2$-dimensional comodules. Equivalently, $A_K:=\lambda^{-1}(\kK\tensor A)$ and $A_2:=\lambda^{-1}(z\kK\tensor A)$. Let $V_2$ denote the span of the $v_i$ and let $W_2$ denote the span of the $w_i$. It is clear that $A_2=V_2\oplus W_2$. Let $\tau:A_2\to A_2$ be the linear map induced by the map that transposes $w_i$ and $v_i$ for all $i$. Notice with this map:
\begin{align*}
    \lambda(v_i)&=\half \Big[(1+y)z\tensor v_i+(1-y)z\tensor \tau(v_i)\Big]\\
    \lambda(w_i)&=\half \Big[x(1-y)z\tensor \tau(w_i)+x(1+y)z\tensor w_i\Big].
\end{align*}

From here, we can describe the isomorphism type of $A_K$ when $A$ is exact. For each groulike element $g\in K$, let $\GA{g}$ denote the coideal subalgebra of $\hkp$ generated by $g$. We can see $\psi:K\times K\to \kk$ defined by 
\begin{equation}\label{eq psi cocy}
    \psi(x^iy^j,x^ky^\ell)=(-1)^{jk}
\end{equation}
defines a cocycle $\psi\in Z^2(\kK,\kk)$. Moreover, the class of $\psi$ is the only non-trivial cohomology class.
\begin{proposition}\label{AK Struct Prop}
    If $(A,\lambda)$ is an AM-exact $\hkp$-comodule algebra then $A_K$ is isomorphic to one of $\kk,\GA{g}$ (for some $g\in K$), $\kk[K], \kkpsiK$ as $\hkp$-comodule algebras.
\end{proposition}
\begin{proof}
    Notice $A^{co\lambda}=V_1^{n_1}$. This implies $n_1=1$. Suppose there exists $g$ such that $n_g>1$. Then multiplication defines a linear map $m:(\kk_g^{n_g})^{\tensor 2}\to \kk_1$. Since $n_g>2$, this map has a kernel. Thus there exists non-zero $a,b\in \kk_g^{n_g}$ such that $ab=0$. Notice then that $a$ is a zero divisor and as such generates a non-trivial ideal. We now notice that this ideal is costable:
    \begin{equation*}
        \lambda(aA)\subset (g\tensor a)(\hkp\tensor A)\subset \hkp\tensor aA.
    \end{equation*}
    This is a contradiction with the right $\hkp$-simplicity of $A$. This implies $n_g$ is at most $1$ for all $g\in\set{x,y,xy}$. Now let $e_g\in \kk_g, e_h\in \kk_h$ be non-zero such that $e_ge_h=0$. Then the right ideal generated by $e_g$ is non-trivial, but by the argument above, $e_gA$ is an $\hkp$-costable right ideal. This is a contradiction so $e_ge_h\neq 0$ for all grouplike elements $g,h$ such that $n_gn_h\neq 0$. 
    
    Notice this implies $A_K\cong \bigoplus_{g\in \Gamma} \kk_g$ as $\hkp$-comodules where $\Gamma\subset K$ is a subgroup. Moreover there exists a map $\psi:\Gamma\times \Gamma\to \kk^\times$ satisfying $e_ge_h=\psi(g,h)e_{gh}$ for all $g,h\in \Gamma$. For $A$ to be associative, $\psi$ is necessarily a group $2$-cocycle. The isomorphism class is independent of the cocycle representative. Since the only subgroup that supports a non-trivial cocycle is $K$ and there is one cohomology class of cocycles, we will take map in \eqref{eq psi cocy}. This shows the above list is exhaustive.
\end{proof}
Proposition \ref{AK Struct Prop} proves $\dim(\lambda^{-1}(g\tensor A)
)\leq 1$ for all $g$. Now that we know the possible isomorphism classes of $A_K$, we can investigate how the presence of two dimensional subcomodules impacts the possible choices of $A_K$. We begin with a lemma that provides a method to detect when an element in $A_2$ is contained in $V_2$ (or $W_2$).

\begin{lemma}\label{V detect lemma}
    Let $v\in A_2$. Then
    \begin{enumerate}
        \item $v\in V_2$ if and only if $((xz)^*\tensor 1)\lambda(v)=0$, and
        \item $v\in W_2$ if and only if $(z^*\tensor 1)\lambda(v)=0$.
    \end{enumerate} 
\end{lemma}
\begin{proof}
    Any vector $v\in A_2$ can be uniquely expressed in the basis $\set{v_i,w_i}$. Thus, there exists $\alpha_i,\beta_i\in\kk$ such that $v=\sum_i \alpha_i v_i+\beta_iw_i$. Then we notice
    \begin{align*}
        ((xz)^*\tensor 1)\lambda(v)=\sum_i \half \beta_i(v_i+w_i).
    \end{align*}
    It is clear that this expression vanishes if and only if $\beta_i=0$ for all $i$. The second statement follows from an analagous computation.
\end{proof}

We see that in $\hkp$, $(\kK )(z\kK)=z\kK$. This implies
\begin{equation*}
    A_KA_2=\lambda^{-1}(\kK\tensor A)\lambda^{-1}(z\kK\tensor A)\subset \lambda^{-1}(z\kK\tensor A)=A_2
\end{equation*}
Similar compuations prove the proceeding lemma.
\begin{lemma}\label{AKA2 Prod Lemma}
    If $A=A_K\oplus A_2$ is an $\hkp$-comodule algebra, then 
    \begin{enumerate}
        \item $A_K$ is an $\hkp$-subcomoodule subalgebra of $A$,
        \item $A_2$ is a left (and right) $A_K$-module under the restriction of the left (resp. right) regular action of $A$ on itself, and 
        \item $A_2A_2\subset A_K$.
    \end{enumerate}
\end{lemma}
The last point in Lemma \ref{AKA2 Prod Lemma} will be crucial in the computations throughout the rest of this section. In particular, if $A$ is exact, then $A_K$ is isomorphic to one of the algebras from Proposition \ref{AK Struct Prop}. We see the product of elements in $A_2$ are constrained in their possible values by the (relatively) small dimension of $A_K$. The following remark formalizes this fact.

\begin{remark}\label{vivj prod str rem}
Let $A$ be an AM-exact $\hkp$-comodule algebra. Let $v_i,v_j\in V_2$ and consider $v_iv_j$. Since $\lambda$ is a morphism of algebras, it is straightforward to compute:
\begin{align*}
    4\lambda(v_iv_j)=&1\tensor (v_iv_j-v_iw_j-w_iv_j-w_iw_j)\\
    &+x\tensor (v_iv_j+v_iw_j-w_iv_j+w_iw_j)\\
    &+y\tensor (v_iv_j-v_iw_j+w_iv_j+w_iw_j)\\
    &+xy\tensor (v_iv_j+v_iw_j+w_iv_j-w_iw_j).
\end{align*}
So it must be the case that:
\begin{align*}
    v_iv_j-v_iw_j-w_iv_j-w_iw_j&=4\alpha_{ij},\\
    v_iv_j+v_iw_j-w_iv_j+w_iw_j&=4\beta_{ij}e_{x},\\
    v_iv_j-v_iw_j+w_iv_j+w_iw_j&=4\gamma_{ij}e_{y},\\
    v_iv_j+v_iw_j+w_iv_j-w_iw_j&=4\delta_{ij}e_{xy},
\end{align*}
for some collection of scalars $\alpha_{ij},\beta_{ij},\gamma_{ij},\delta_{ij}\in\kk$. Of course it could be the case that $A_g=0$ for some $g$ and no such $e_g$ can be chosen. The statements above still hold, we simply enforce the vanishing of the appropriate scalars. This system has a unique solution:
\begin{align*}
    v_iv_j&=\alpha_{ij}+\beta_{ij}e_x+\gamma_{ij}e_y+\delta_{ij}e_{xy},\\
    v_iw_j&=-\alpha_{ij}+\beta_{ij}e_x-\gamma_{ij}e_y+\delta_{ij}e_{xy},\\
    w_iv_j&=-\alpha_{ij}-\beta_{ij}e_x+\gamma_{ij}e_y+\delta_{ij}e_{xy},\\
    w_iw_j&=-\alpha_{ij}+\beta_{ij}e_x+\gamma_{ij}e_y-\delta_{ij}e_{xy}.
\end{align*}
This fact will be referenced frequently in the coming proofs. We will refer to the the collection of scalars $\alpha_{ij}, \beta_{ij}, \gamma_{ij}, \delta_{ij}\in\kk$ as the structure constants of $A$ with respect to the basis $\set{v_i}\subset V_2$. We should note that the structure constants depend on the choice of basis for $V_2$ but if all of the structure constants vanish with respect to one basis, then all of the structure constants vanish for every choice of basis for $V_2$. 
\end{remark}
We notice that the constants are defined with respect to a basis for $V_2$, not $A_2$. This follows because a basis for $V_2$ automatically defines a basis for $W_2$ because each basis vector for $V_2$ generates pairwise disjoint 2 dimensional subcomodules. The following is a useful method for determining when $A=A_K$.
\begin{lemma}\label{Square Zero Lemma}
    Let $A$ be an exact $\hkp$-comodule algebra. Suppose $A_2A_2=0$ then $A_2=0$.
\end{lemma}
\begin{proof}
    Lemma \ref{AKA2 Prod Lemma} shows that $A_2$ is a proper ideal and it follows directly that $A_2$ is costable. Since $A$ is exact, $A_2=0$.
\end{proof}
We will note that if all of the structure constants vanish for $A$, then $A_2A_2=0$ and thus $A=A_K$. This is the most common technique that we will use to show $A=A_K$.

%% file: Sections/5.H8_sub/Small_Dim.tex
\begin{proposition}\label{AK Gen X/Y Prop}
    If $A$ is an exact $\hkp$-comodule algebra and $A_K$ is isomorphic to one of $\kk, \GA{x}, \GA{y}$, then $A=A_K$.
\end{proposition}
\begin{proof}
    Let $A$ be as given.
    
    Suppose $A_K\cong \kk$ and that $A_2\neq 0$. That is, there exists $n_2\neq 0$ and $\set{v_i,w_i}_{i=1}^{n_2}$ a basis for $A_2$. Let $1\leq i,j,k\leq n_2$, notice by remark \ref{vivj prod str rem}:
    \begin{equation*}
        \alpha_{ij}w_k=v_iv_jw_k=-\alpha_{jk}v_i.
    \end{equation*}
    Since $w_k,v_i$ are $\kk$-linearly independent, $\alpha_{ij}=0$ for all $i,j$. Thus by Lemma \ref{Square Zero Lemma}, $A_2=0$ and thus $A=A_K$.

    Now suppose $A_K\cong \GA{x}$ and $A_2\neq 0$.  Notice Remark \ref{vivj prod str rem} implies $v_iv_j=-w_iv_j$ so $(v_i+w_i)v_j=0$ for all $i,j$. Thus
    \begin{align*}
        0&=w_j(v_i+w_i)v_j=(w_jv_i+w_jw_i)v_j=-2\alpha_{ji}v_j,\\
        0&=e_xv_j(v_i+w_i)v_j=e_x(v_jv_i+v_jw_i)v_j=2\beta_{ji}v_j.
    \end{align*}
    This implies $\alpha_{ij}=\beta_{ij}=0$ for all $i,j$. Thus by Lemma \ref{Square Zero Lemma}, $A=A_K$. Similarly, if $A_K=\GA{y}$ and $A_2\neq 0$, then $v_iv_j+v_iw_j=0$ for all $i,j$. We then notice
    \begin{align*}
        0&=v_i(v_j+w_j)w_i=v_i(v_jw_i+w_jw_i)=-2\alpha_{ji}v_i,\\
        0&=v_i(v_j+w_j) v_i e_y =v_i(v_jv_i+w_jv_i) e_y=2\gamma_{ji} v_i.
    \end{align*}
    This implies $\alpha_{ij}=\gamma_{ij}=0$ for all $i,j$. Thus by Lemma \ref{Square Zero Lemma}, $A=A_K$.
\end{proof}

We establish some notation. If $A_K=\GA{g}$ for some $g\neq 1$, then we take a basis for $A_K$ given by $\set{1,e_g}$ such that $e_g^2=1$ and $\lambda(e_g)=g\tensor e_g$. If $A_K$ is isomorphic to $\kK$ or $\kkpsiK$, then we take a basis for $A_K$ by $\set{1,e_x,e_y,e_{xy}}$ such that $e_x^2=e_y^2=1$ and $\lambda(e_g)=g\tensor e_g$. This choice can be made regardless of the deformation by $\psi$ and is unique up to multiplication by $-1$. 

If $g\in K$ such that $\lambda^{-1}(g\tensor A)=1$, then define $L_g, R_g\in \End_\kk(A)$ given by left and right (respectively) multiplication by $e_g$. We notice each such map is invertible and the next lemma shows how this action interacts with the two dimensional subcomodules.


\begin{lemma}\label{L and R V and W Lemma}
    Let $A$ be an exact $\hkp$-comodule algebra. Then (when defined),
    \begin{enumerate}
        \item $L_y(V_2)=R_x(V_2)=V_2$, and
        \item $L_x(V_2)=R_y(V_2)=W_2$.
        \item $L_{xy}(V_2)=R_{xy}(V_2)=W_2$.
    \end{enumerate}
\end{lemma}
\begin{proof}
    Let $v_i\in V_2$ and consider $L_y(v_i)$:
    \begin{align*}
        ((xz)^*\tensor 1)\lambda(L_y(v_i))&=((xz)^*\tensor 1)\lambda(e_yv_i)\\
        &=(xz)^*\tensor 1)\paren{(y\tensor e_y)\half \Big[(1+y)z\tensor v_i+(1-y)z\tensor w_i\Big]}\\
        &=(xz)^*\tensor 1)\paren{\half \Big[(1+y)z\tensor e_yv_i-(1-y)z\tensor e_yw_i\Big]}\\
        &=0.
    \end{align*}
    Thus by Lemma \ref{V detect lemma}, $L_y(V_2)\subset V_2$. Since $L_y^2=1$, $L_y(V_2)=V_2$. Now consider $L_x(w_i)$:
    \begin{align*}
        ((xz)^*\tensor 1)\lambda(L_x(w_i))&=((xz)^*\tensor 1)\lambda(e_xw_i)\\
        &=((xz)^*\tensor 1)\paren{\half \Big[(1-y)z\tensor e_xv_i+(1+y)z\tensor e_xw_i\Big]}\\
        &=0.
    \end{align*}
    Thus $L_x(W_2)\subset V_2$. Again, by invertibility this implies $L_x(W_2)=V_2$ and 2 above holds by applying $L_x$ again. Similar computations can be carried out to prove the remaining statments.
\end{proof}

\begin{proposition}
    If $A_K=\GA{xy}$ then $n_2$ is at most 1.
\end{proposition}
\begin{proof}
    Suppose $n_2> 1$. Consider the product $v_jv_jv_i$ for $i\neq j$. Notice:
    \begin{equation*}
        \alpha_{jj}v_i+\delta_{jj}L_{xy}(v_i)=(\alpha_{jj}+\delta_{jj}e_{xy})v_j=v_jv_jv_i=v_j(\alpha_{ji}+\delta_{ij}e_{xy})=\alpha_{ji}v_j+\delta_{ij}R_{xy}(v_j)
    \end{equation*}
    By the Lemma \ref{L and R V and W Lemma}, $L_{xy}(v_i), R_{xy}(v_j)\in W_2$ because $L_{xy}$ is proportional to $L_xL_y$ (similarly for $R_{xy}$). Thus by linear independence of $v_i,v_j$, it must be the case that $\alpha_{jj}=\alpha_{ij}=0$ for all $i,j$. Thus:
    \begin{align*}
        v_iv_j&=\delta_{ij}e_{xy},\\
        v_iw_j&=\delta_{ij}e_{xy},\\
        w_iv_j&=\delta_{ij}e_{xy},\\
        w_iw_j&=-\delta_{ij}e_{xy},
    \end{align*}
    for some collection $\delta_{ij}\in\kk$. Notice that $v_i(v_j-w_j)=0$. Thus:
    \begin{equation*}
        0=v_i(v_j-w_j)w_ke_{xy}=v_i(v_jw_k-w_jw_k)e_{xy}=2\delta_{jk}v_i
    \end{equation*}
    Thus $\delta_{jk}=0$ for all $j,k$. Thus by a Lemma \ref{Square Zero Lemma}, $A=A_K$, a contradiction.
\end{proof}
Now we will show that there exists an $\hkp$-comodule algebra $A$ such that $A_K=\GA{xy}$ and $A\neq A_K$.

\begin{proposition}\label{gamma coideal Prop}
    There is one isomorphism class of AM-exact $\hkp$-comodule algebra with $A_K=\GA{xy}$ and $A\neq A_K$.
\end{proposition}
\begin{proof}
    If such an algebra exists, it must be isomorphic to $\GA{xy}\oplus X$ as $H_8$-comodules. Thus we take a basis $\set{1,e_{xy},v,w}$ for $A$. Notice by Remark \ref{vivj prod str rem}, $v$ and $w$ commute. By Lemma \ref{L and R V and W Lemma}, there exists nonzero $\gamma\in\kk$ such that $L_{xy}(v)=\gamma w$. Notice further that $v^2\in A_K$ and as such commutes with $e_{xy}$. This implies
    \begin{equation*}
        \gamma wv=(e_{xy}v)v=v(ve_{xy})=vR_{xy}(v)
    \end{equation*}
    and since $R_{xy}(v)$ is a scalar multiple of $w$, and $w$ and $v$ commute, $R_{xy}(v)=\gamma w$.

    We now notice that
    \begin{align*}
        \lambda(e_{xy}v)&=(xy\tensor e_{xy})\half \Big[(1+y)z\tensor v+(1-y)z\tensor w\Big]\\
        &=\half \Big[x(1+y)z\tensor e_{xy}v+ x(1-y)z\tensor -e_{xy}w\Big]\\
        &=\half \Big[x(1+y)z\tensor \gamma w+ x(1-y)z\tensor -\gamma^{-1}v\Big]
    \end{align*}
    but of course this must equal $\lambda(\gamma w)=\half \Big[x(1-y)z\tensor \gamma v+x(1+y)z\tensor \gamma w\Big]$. Thus $\gamma=-\gamma^{-1}$ and so $\gamma$ is necessarily a primitive fourth root of unity. Let $\alpha,\delta\in \kk$ be scalars as in Remark \ref{vivj prod str rem}. Then we notice $e_{xy}v^2=\gamma wv$, which gives the equality
    \begin{equation*}
        e_{xy}(\alpha+\delta e_{xy})=e_{xy}v^2=\gamma wv=\gamma (-\alpha+\delta e_{xy})
    \end{equation*}
    and implies $\delta=-\alpha\gamma$. By scaling $v$ and $w$, we can assume $\alpha=1$. We will denote this algebra $A_{xy}^\gamma$. Notice the choice of $\gamma$ is unique up to $-1$ but we see that the linear map $f:A_{xy}^\gamma\to A_{xy}^{-\gamma}$ defined by:
    \begin{align*}
        f(1)&=1,&f(v)&=-v,\\
        f(e_{xy})&=-e_{xy},&f(w)&=-w.
    \end{align*}
    defines an $H$-comodule algebra isomorphism $A_{xy}^\gamma\cong A_{xy}^{-\gamma}$. Thus the isomorphism class of $A_{xy}^\gamma$ is independent of the choice of primitive fourth root of unity. To see that $A_{xy}^\gamma$ is right $H_8$-simple, it follows from direct computation that the coideal subalgebra $S$ generated by $z+\gamma yz$ is isomorphic to $\GA{xy}\oplus X$ as comodules. The argument above implies $S\cong A_{xy}^\gamma$ as $H_8$-comodule algebras and coideal subalgebras are AM-exact.
\end{proof}

%% file: Sections/5.H8_sub/Larger_Cases.tex
The last two cases to consider are when $A_K=\kK$ or $A_K=\kkpsiK$. We notice that regardless of the isomorphism type of $A_K$, $R_x$ and $L_y$ commute by associativity and square to the identity. This defines a $\kK$-module structure on $A$ where $x$ acts by $R_x$ and $y$ acts by $L_y$. We notice that Lemma \ref{L and R V and W Lemma} implies $A_2, V_2, W_2$ are submodules under this action. Thus we can take a basis $\set{v_i}_{i=1}^{n_2}\subset V_2$ such that each $v_i$ generates a simple $\kK$-module. That is, $L_y(v_i), R_x(v_i)\in \set{\pm v_i}$ for all $i$. Therefore, $V_2=\bigoplus_{a,b\in\set{\pm 1}} V_2(\mu_{a,b})$ where 
\begin{equation*}
    V_2(\mu_{a,b}):=\set{v\in V_2:R_x(v)=av, L_y(v)=bv}.
\end{equation*}
Although we will investigate each case separately, we first collect some general facts.

\begin{lemma}\label{Tau Dim Lem}
    For all $a,b\in\set{\pm 1}$,  $\dim(V(\mu_{a,b}))=\dim(W(\mu_{-a,-b}))$.
\end{lemma}
\begin{proof}
    Suppose $v\in V(\mu_{a,b})$. Then
    \begin{align*}
        \lambda(L_y(v))&=(y\tensor e_y)\half \Big[(1+y)z\tensor v+(1-y)z\tensor \tau(v)\Big]\\
        &=\half \Big[(1+y)z\tensor bv+(1-y)z\tensor -L_y(\tau(v))\Big]
    \end{align*}
    and this is equal to $\lambda(b v)=\half \Big[(1+y)z\tensor bv+(1-y)z\tensor b\tau(v)\Big]$. Thus $L_y(\tau(v))=-b\tau(v)$. A similar computation shows that $R_x(\tau(v))=-av$. The result then follows from bijectivity of $L_y, R_x$ and $\tau$.
\end{proof}

\subsection{The Group Algebra}
Suppose $A$ is an AM-exact left $\hkp$-comodule algebra and $A_K\cong \kK$. We will abuse notation and take $e_x=x, e_y=y$. We notice by associativity that if $v_i\in V(\mu_{a,b})$ and $v_j\in V(\mu_{a',b'})$, then
\begin{equation*}
    v_iv_j=\alpha_{ij}(1+by)(1+a'x).
\end{equation*}
We now show that any product $v_iv_j$ vanishing leads to contradiction.
\begin{lemma}\label{Grp ij0 Lem}
    If $v_iv_j=0$ for any $i,j$, then $A=A_K$.
\end{lemma}
\begin{proof}
    Suppose there exists $v_i\in V(\mu_{a,b}), v_j\in V(\mu_{a',b'})$ such that $v_iv_j=0$, then $\alpha_{ij}=0$. Notice then for any $k$, $w_kv_iv_j=v_kv_iv_j=0$, but by associativity:
    \begin{align*}
        0=v_kv_iv_j&=\alpha_{ki}(1+by)(1+a'x)v_j=\alpha_{ki}(1+bb')(v_j+a'xv_j)\\
        0=w_kv_iv_j&=-\alpha_{ki}(1-by)(1+a'x)v_j=-\alpha_{ki}(1-bb')(v_j+a'xv_j)
    \end{align*}
    Since $v_j, xv_j$ are linearly independent, (see Lemma \ref{L and R V and W Lemma}) it must be the case that both $\alpha_{ki}(1+bb')$ and $-\alpha_{ki}(1-bb')$ vanish. Thus $\alpha_{ki}=0$. By the same argument, for any $\ell$, $\alpha_{ki}=0$ implies $\alpha_{\ell k}=0$. Since $\ell,k$ are arbitrary, $A=A_K$ by Lemma \ref{Square Zero Lemma}.
\end{proof}

\begin{lemma}\label{Grp small dim Lem}
    For all $a,b\in\set{\pm 1}$, $V(\mu_{a,b})\leq 1$.
\end{lemma}
\begin{proof}
    Suppose on the contrary, there exists  linearly independent $v_1,v_2\in V(\mu_{a,b})$. Then $v_1^2$ and $v_1v_2$ are linearly dependent and there exists non-zero scalar $c\in\kk$ such that $v_1^2-cv_1v_2=0$. We notice the vectors $v_1'=v_1, v_2'=v_1-cv_2$ extends to a basis where $v_1'v_2'=0$ and by Lemma \ref{Grp ij0 Lem} we derive a contradiction.
\end{proof}
We notice if $0\neq v\in V(\mu_{a,b})$, then $0\neq xv\in W_2$ by Lemma \ref{L and R V and W Lemma}. Further, since $x$ commutes with $y$, $xv\in W(\mu_{a,b})$. We notice that this observation, along with Lemma \ref{Grp small dim Lem}, implies that if $A_K=\kK$, then $A$ is isomorphic as a comodule to one of 
\begin{itemize}
    \item $\kK$,
    \item $\kK\oplus A_2^+$, ($\kK\oplus A_2^-$), or 
    \item $\kK\oplus A_2^+\oplus A_2^-$
\end{itemize}
where
\begin{equation*}
    A_2^+:=\bigoplus_{ab=1} V(\mu_{a,b})\oplus W(\mu_{a,b}),\text{ and}, A_2^-:=\bigoplus_{ab=-1} V(\mu_{a,b})\oplus W(\mu_{a,b}).
\end{equation*}
We see now that the last case is impossible.
\begin{lemma}\label{Grp 8d Lem}
    If $A$ is an AM-exact $\hkp$-comodule algebra with $A\neq A_K$, then $A$ is at most eight dimensional.
\end{lemma}
\begin{proof}
    In particular, we will see that if $A\cong A_K\oplus A_2^+\oplus A_2^-$, then we derive a contradiction.

    Suppose $A=A_K\oplus A^+_2\oplus A^-_2$. Then there exists non-zero $v_1\in V(\mu_{1,1})$ and $v_2\in V(\mu_{-1,1})$. Now
    \begin{align*}
        (v_1^2)v_2&=\alpha_{11}(1+x+y+xy)v_2=\alpha_{11}(v_2+xv_2-v_2-xv_2)=0,\\
        v_1(v_1v_2)&=\alpha_{12} v_1(1+x+y+xy)=\alpha_{12}(2v_1+2v_1y)
    \end{align*}
    and by associativity, these expressions must be equal. Since $v_1y\subset W$ by Lemma \ref{L and R V and W Lemma}, $v_1+v_1y\neq 0$ and so $\alpha_{12}=0$. Thus by Lemma \ref{Grp ij0 Lem}, $A=A_K$, a contradiction.
\end{proof}

\begin{proposition}\label{Grp H8 Prop}
    If $A$ is an AM-exact $\hkp$-comodule algebra with $A_K=\kK$ and $A\neq A_K$, then $A$ is isomorphic to $\hkp$ as $\hkp$-comodule algebras.
\end{proposition}
\begin{proof}
    Lemma \ref{Grp 8d Lem} implies $A$ is either of the necessarily of the form $A_K\oplus A_2^+$ or $A_K\oplus A_2^-$.
    
    Suppose $A=A_K\oplus A_2^+$ with $v_1\in V(\mu_{1,1}), v_2\in V(\mu_{-1,-1}), w_1\in W(\mu_{-1,-1}), w_2\in W(\mu_{1,1})$. We notice that the algebra structure on $A$ is determined by the choices of $\alpha_{11},\alpha_{12},\alpha_{21},\alpha_{22}$. Notice without loss of generality we can assume $\alpha_{11}=1$ and $\alpha_{12}=-1$ by scaling $\set{v_1,w_1}$ and $\set{v_2,w_2}$. We make this choice so that $v_1^2=v_1w_2$.
    \\[5px]
    Consider the following computations:
    \begin{align*}
        (v_1v_2)v_1&=-(1+y)(1-x)v_1=2(1-x)v_1\\
        v_1(v_2v_1)&=\alpha_{21}v_1(1-y)(1+x)=\alpha_{21}2v_1(1-y)
    \end{align*}
    Since $v_1$ is linearly independent from both $xv_1$ and $v_1y$ (Lemma \ref{L and R V and W Lemma}), it must be the case that $\alpha_{21}=-1$. We also see that $xv_1=v_1y$. In a similar manner, a comparison of $(v_2v_1)v_2$ and $v_2(v_1v_2)$ shows $xv_2=v_2y$. Further, Lemma \ref{L and R V and W Lemma} implies $xv_1=v_1y=c w_2$ for some non-zero $c\in \kk$. Now we notice:
    \begin{equation*}
        v_1v_2=v_1^2=(v_1x)(xv_1)=c v_1w_2
    \end{equation*}
    and so $c=1$. Now we see that
    \begin{equation*}
        w_2^2=(v_1y)(xv_1)=(v_1x)(yv_1)=v_1^2.
    \end{equation*}
    Notice this implies $\alpha_{22}=-1$ (see Remark \ref{vivj prod str rem}). Now that we have determined the structure coefficients, we can see the map $f:A\to H_8$ defined by $f(g)=g$ for $g\in K$ and
    \begin{align*}
        f(v_1)&=yz+z& f(v_2)&=yz-z\\
        f(w_1)&=xyz-xz& f(w_2)&=xyz+xz
    \end{align*}
    is an isomorphism of $\hkp$ comodule algebras.

    We notice that the choice of $x$ (or $y$) is only unique up to $\pm 1$. By exchanging $x$ for $-x$, we see that $A_K\oplus A_2^-\cong A_K\oplus A_2^+$ as $H$-comodules.
\end{proof}

\subsection{The Twisted Group Algebra}
Suppose $A$ is an AM-exact $\hkp$-comodule algebra and $A_K\cong \kkpsiK$. Again, we take a basis $\set{e_g}_{g\in K}$ such that $e_ge_h=\psi(g,h)e_{gh}$ where $\psi(x^iy^j,x^ky^\ell)=(-1)^{jk}$. We notice if $v_i\in V(\mu_{a,b})$ and $v_j\in V(\mu_{a',b'})$, then
\begin{equation*}
    v_iv_j=\alpha_{ij}(1+be_y)(1+a'e_x)
\end{equation*}
Now we will show the vanishing of $v_iv_j$ leads to contradiction
\begin{lemma}
    If $v_iv_j=0$ for any $i,j$ then $A=A_K$.
\end{lemma}
\begin{proof}
    Take $v_i\in V(\mu_{a,b}), v_j\in V(\mu_{c,d})$, and $v_k\in V(\mu_{e,f})$. Then $v_kv_iv_j=0$. By associativity, we have:
    \begin{align*}
        0&=v_kv_iv_j\\
        &=\alpha_{ki}(1+fe_y)(1+ae_x)v_j\\
        &=\alpha_{ki}(1+fe_y+ae_x-afe_xe_y)v_j\\
        &=\alpha_{ki}(1+cf+ae_x-acfe_x)v_j\\
        &=\alpha_{ki}(1+cf)v_j+\alpha_{ki}a(1-cf)e_xv_j
    \end{align*}
    Notice one of $(1+cf)$ and $(1-fc)$ fails to vanish. In either case it is necessary for $\alpha_{ki}=0$. A similar argument would tell us that $\alpha_{ki}=0$ implies $\alpha_{\ell k}=0$. Since $\ell$ and $k$ were arbitrary, all of the structure constants vanish and so $A=A_K$ by Lemma \ref{Square Zero Lemma}.
\end{proof}
We now notice that if $v_i\in V(\mu_{a,b})$, then $v_ie_y\in W$ by Lemma \ref{L and R V and W Lemma}. Further we see:
\begin{align*}
    (v_ie_y)e_x&=-(v_ie_x)e_y=-av_ie_y\\
    e_y(v_i e_y)&=(e_y v_i)e_y=bv_ie_y
\end{align*}
which implies $\dim(V(\mu_{a,b}))=\dim(V(\mu_{-a,b}))$ for all $a,b$. This fact, together with Lemma \ref{Tau Dim Lem}, implies $\dim(V(\mu_{a,b}))=\dim(W(\mu_{c,d}))$ for all $a,b,c,d$.

\begin{proposition}\label{Twist no ext Prop}
    If $A$ is an AM-exact $\hkp$-comodule algebra with $A_K\cong \kkpsiK$, then $A=A_K$.
\end{proposition}
\begin{proof}
    Suppose $A\neq A_K$. Then $\dim(V(\mu_{a,b}))>1$ for all $a,b$. Let $v_1$ be a non-zero vector in $V(\mu_{1,1})$ and $v_2$ a non-zero vector in $V(\mu_{-1,1})$. Then we see that $v_1^2$ and $v_2v_1$ are linearly dependent. That is there exists a non-zero scalar $c\in\kk$ such that $v_1^2-cv_2v_1=0$. The structure constants for any basis extending the linearly independent set $\set{v_1, v_1-cv_2}$ must all vanish and so Lemma \ref{Square Zero Lemma} implies the result.
\end{proof}
This completes the classification of AM-exact $\hkp$-comodule algebras.

\begin{theorem}\label{hkp class thm}
    If $A$ is an AM-exact $\hkp$-comodule algebra then $A$ is isomorphic to either a coideal subalgebra of $\hkp$ or the twisted group algebra $\kkpsiK$.
\end{theorem}
\begin{proof}
    This follows from Proposition \ref{AK Struct Prop}, Proposition \ref{AK Gen X/Y Prop}, Proposition \ref{gamma coideal Prop}, Proposition \ref{Grp H8 Prop}, and Proposition \ref{Twist no ext Prop}.
\end{proof}

\subsection{$\hkp$-Equivariant Morita Equivalence}\label{HKP Morita Sub}

Recall ${}_\hkp\M$ is semisimple with five non-isomorphic simple objects. There are four one dimensional $\hkp$-modules,  $\set{\kk_\zeta}_{\zeta:\zeta^4=1}$ where $z$ acts by $\zeta$ on $\kk_\zeta$. The last simple object is the two dimensional representation which we will denote $W$. We can easily see that some of the algebras above are not Morita equivalent by observing the actions of the simple objects on the corresponding module categories.

In particular, if $A\simeq_\hkp B$ then $A$ and $B$ are Morita equivalent as $\kk$-algebras. Thus we will only investigate pairs from Theorem \ref{hkp class thm} that are Morita equivalent as $\kk$-algebras. In particular, we notice the algebra $\kkpsiK$ is isomorphic to $M_2(\kk)$. The isomorphism is realized by $e_x\mapsto e_{11}-e_{22}$ and $e_y\mapsto e_{12}+e_{21}$. In particular, ${}_\kkpsiK\M$ is semisimple with one simple two dimensional representation and $\kkpsiK$ is Morita equivalent as $\kk$-algebras to $\kk$. The next proposition shows that $\kk\simeq_{\hkp}\kkpsiK$.

\begin{proposition}\label{Twist Morita Prop}
    The algebras $\kk, \kkpsiK$ are $\hkp$-equivariant Morita equivalent.
\end{proposition}
\begin{proof}
    Consider the two dimensional $\hkp$ comodule $V$ given above with basis $v,w$. Define a right $\kkpsiK$-module structure on $V$ by:
    \begin{align*}
        ve_x&=v&we_x&=-w\\
        ve_y&=w&we_y&=v.
    \end{align*}
    It is straightforward to show this make $V$ a right $\kkpsiK$-module. We notice with this structure we have the following computations:
    \begin{align*}
        ((yz+z)\tensor v)(x\tensor e_x)&=(yz+z)x\tensor ve_x=(yz+z)\tensor v\\
        ((yz-z)\tensor w)(x\tensor e_x)&=(z-yz)\tensor we_x=(yz-z)\tensor w\\
        ((yz+z)\tensor v)(y\tensor e_y)&=(yz+z)y\tensor ve_y=x(yz+z)\tensor w\\
        ((yz-z)\tensor w)(y\tensor e_y)&=(yz-z)y\tensor we_y=x(yz-z)\tensor v
    \end{align*}
    and
    \begin{align*}
        (x(yz-z)\tensor v)(x\tensor e_x)&=x(z-yz)\tensor ve_x=x(z-yz)\tensor v\\
        (x(yz+z)\tensor w)(x\tensor e_x)&=x(yz+z)x\tensor we_x=x(z+yz)\tensor -w\\
        (x(yz-z)\tensor v)(y\tensor e_y)&=x(yz-z)y\tensor ve_y=(yz-z)\tensor w\\
        (x(yz+z)\tensor w)(y\tensor e_y)&=x(yz+z)y\tensor we_y=(yz+z)\tensor v
    \end{align*}
    The above computations can be used to show $\lambda:V\to \hkp\tensor V$ is a morphism of right $\kkpsiK$-modules.
    Thus $V\in {}^{\hkp}\M_{\kkpsiK}$ and $\End_{\kkpsiK}(V)$ is a $\hkp$-equivariant Morita equivalent $\hkp$-comodule algebra. Notice any $f\in \End_{\kkpsiK}$ must preserve the $e_x$ eigenspaces and since $V$ is cyclic as a right $\kkpsiK$ module, $f$ is entirely determined by the image of $v$. Therefore, $f(v)=\alpha v$ for some $\alpha\in\kk$ which implies $\End_{\kkpsiK}(V)\cong \kk$ and the result holds.
\end{proof}
Notice this implies every AM-exact $H$-comodule algebra is $\hkp$-equivariant Morita equivalent to a coideal subalgebra of $\hkp$. We now investigate the $\hkp$-equivariant Morita equivalence of coideal subalgebras of $\hkp$.

\begin{proposition}\label{X/Y Morita Prop}
    There is a $\hkp$-equivariant Morita equivalence $\GA{x}\simeq_\hkp \GA{y}$.
\end{proposition}
\begin{proof}
    These algebras are isomorphic so they clearly generate the same abelian category. Now we notice $\GA{x}$ is isomorphic to an $H$-costable subalgebra of $\kkpsiK$ by the map $x\mapsto e_x$. Thus the $\kkpsiK$-module $V$ in the proof of Proposition \ref{Twist Morita Prop} is an object in ${}^H\M_{\GA{x}}$. We notice $V$ is the direct sum of the two non-isomorphic simple representations of $\GA{x}$. This implies $\dim(\End_{\GA{x}}(V))=2$. This implies $\End_{\GA{x}}(V)$ is isomorphic (as a comodule) to $\kk_1\oplus \kk_g$ for some grouplike element $g$. In particular, there exists an element $F$ such that $F^2=1$ and $\lambda(F)=g\tensor F$. Notice such a map must preserve the subspaces generated by $x$ and $y$. Thus $F(v)=\alpha v$ and $F(w)=\beta v$ for some $\alpha,\beta\in\set{\pm 1}$. The case when $\alpha=\beta=1$ is the identity and $\alpha=\beta=-1$ is negative of the identity. Thus we consider the map $F(v)=v, F(w)=-w$. Since $V\in {}^\hkp_{\End_{\GA{x}}(V)}\M_{\GA{x}}$:
    \begin{equation*}
        \lambda(F(p))=\lambda(F)\lambda(p)
    \end{equation*}
    for all $p\in V$. In particular, when $p=v$ we have:
    \begin{align*}
        \half \Big[(1+y)z\tensor v+(1-y)z\tensor w\Big]&=(g\tensor F)\half \Big[(1+y)z\tensor v+(1-y)z\tensor w\Big]\\
        &=\half \Big[g(1+y)z\tensor F(v)+g(1-y)z\tensor F(w)\Big]\\
        &=\half \Big[g(1+y)z\tensor v+g(1-y)z\tensor -w\Big]
    \end{align*}
    which we notice only yields equality if and only if $g=y$. Thus $\lambda(F)=y\tensor F$ and $\End_{\GA{x}}(V)\cong \GA{y}$.
\end{proof}
An interesting note to make is that $\GA{x}\not\simeq_{\kK}\GA{y}$. 
\begin{remark}\label{Morita Eq as K rem}
    Although $\GA{x}$ and $\GA{y}$ are $\hkp$-equivariant Morita equivalent, but $\GA{x}\not\simeq_{\kK}\GA{y}$.
\end{remark}
\begin{proof}
    Suppose there exists $P\in {}^{\kK}\M_{\GA{y}}$ such that $\End_{\GA{y}}(P)\cong \GA{x}$ as $H$-comodule algebras. Notice by Theorem \ref{4.2iii_SKR}, any $P={}^{\kK}\M_{\GA{y}}$ is free as a $\GA{y}$-module. Thus $\End_{\GA{y}}(P)\cong \GA{x}$ (as $\kk$-algebras) if and only if $P$ is two dimensional. In this case, $P$ is also the direct sum of two, one dimensional $\kK$-comodules. That is, there exists an element $b_g$ such that $\lambda(b_g)=g\tensor b_g$ and $\lambda (b_gy)=gy\tensor b_gy$. Notice $\End_{\GA{y}}(P)$ is isomorphic as a comodule to the direct sum of two simple subcomodules. The first is spanned by the identity map. The second is spanned by a morphism satisfying $T^2=1$. Thus $T$ (or $-T$) must fix $(b_g+b_gy)$ and negate $(b_g-b_gy)$.  Notice 
    \begin{align*}
        T(b_g)=T(\half (b_g+b_gy)+\half(b_g-b_gy))=\half (b_g+b_gy)-\half(b_g-b_gy))=b_gy
    \end{align*}
    Since $P\in {}^{\kK}\M_{\GA{y}}$, 
    \begin{equation*}
        gy\tensor b_gy=\lambda(T(b_g))=(h\tensor T)(g\tensor b_g)=hg\tensor b_gy
    \end{equation*}
    Thus $gy=hg$ and so $h=y$. Thus $\End_{\GA{y}}(P)\cong \GA{y}$. That is, there is no object $P\in {}^{\kK}\M_{\GA{y}}$ such that $\End_{\GA{y}}(P)\cong \GA{x}$ as $\kK$-comodule algebras.
\end{proof}

Remark \ref{Morita Eq as K rem} shows that the two dimensional simple comodule appearing in ${}^\hkp\M$ that is not induced from an object ${}^\kK\M$ is necessary to realize the equivalence $\M_{\GA{x}}\simeq_{\hkp}\M_{\GA{y}}$.

\begin{proposition}\label{X/XY Noneq Prop}
    The algebras generated by $x$ and $xy$ are Morita equivalent as $\kk$-algebras but are not $\hkp$-equivariant Morita equivalent.
\end{proposition}
\begin{proof}
    Consider $\GA{x},\GA{xy}$ denote these algebras and let $X_\alpha$ for $\alpha=\pm 1$ be the simple representations of $\GA{x}$ and let $U_\alpha$ be the simple representations of $\GA{xy}$. By means of contradiction, suppose $F:{}_{\GA{x}}\M \to {}_{\GA{xy}}\M$ defines a $\Rep(\hkp)$-module category equivalence. Then for all $\zeta$, $F(\kk_\zeta\tensor X_\alpha)$ must be isomorphic to $\kk_\zeta \tensor F(X_\alpha)$ for all $\zeta^4=1$. Notice though that $\kk_\zeta\tensor X_\alpha\cong X_{\zeta^2\alpha}$ and $\kk_\zeta\tensor U_\alpha\cong U_\alpha$ for all $\zeta,\alpha$. Thus the action of $\kk_\zeta$ permutes the simple objects of ${}_{\GA{x}}\M$ when $\zeta^2=-1$ but the action of $\kk_\zeta$ fixes the isomorphism class of simple objects in ${}_{\GA{xy}}\M$ for all $\zeta$. Thus no such functor can exist.
\end{proof}
A similar argument can show that $\kK\not\simeq_\hkp A_{xy}^\gamma$ despite $\kK\cong A_{xy}^\gamma$ as $\kk$-algebras. The isomorphism follows from the fact that these are both four dimensional commutative semisimple (by exactness) $\kk$-algebras.
\begin{proposition}
    The algebras $\kk[K],A^\gamma_{xy}$ are not $\hkp$-equivariant Morita equivalent.
\end{proposition}
\begin{proof}
    Let $\kk_{a,b}\in \Rep(\kK)$ be the simple representations of $\kK$ where $x$ acts by $a$ and $y$ acts by $b$ for $(a,b)\in\set{(\pm 1,\pm 1)}$. Let $\kk_{c,\sigma_c}$ be the simple representation of $A^\gamma_{xy}$ with $e_{xy}$ acting by $c=\pm 1$ and $v$ acting by $\sigma_c$, a square root of $1-cq$.
    \\[5px]
    We can see that $\kk_{-1}\tensor \kk_{a,b}\cong \kk_{a,b}$ for all $a,b$ but $\kk_{-1}\tensor \kk_{c,\sigma_c}\cong \kk_{c,-\sigma_c}$ for all $c,\sigma_c$. Thus $\kK\not\simeq_{\hkp}A^\gamma_{xy}$.
\end{proof}

\begin{theorem}\label{Morita Deg Z Thm}
    If $\mathcal{M}$ is an exact indecomposable module category over $\Rep(\hkp)$ then $\mathcal{M}$ is equivalent to ${}_A\mathcal{M}$ for some left coideal subalgebra $A\subset \hkp$.
\end{theorem}
\begin{proof}
    The classification of isomorphism class of AM-exact $\hkp$-comodule algebras follows from Theorem \ref{hkp class thm}. Then, Proposition \ref{Twist Morita Prop}, Proposition \ref{X/Y Morita Prop}, and Proposition \ref{X/XY Noneq Prop} prove several (in)equivalences. We notice all other pairs unmentioned in these propositions are between $\kk$-algebras that are not (classically) Morita equivalent.
\end{proof}

The classification of ${}_\hkp\M$-module categories was computed in \cite{EKW} via $\hkp$-module algebras. In particular, Theorem \ref{Morita Deg Z Thm} is analogous to Theorem 5.23 in \cite{EKW}. The benefit of this approach is that we can extend these results to the non-semisimple coradically-graded Hopf algebras with coradical $\hkp$ (see \cite{SHI}). In particular, $\hkp$ satisfies the hypothesis of Corollary \ref{All Coideal cor} which implies the following result.
\begin{proposition}
    If $H$ is a finite-dimensional coradically graded Hopf algebra with $H_0=\hkp$, then every Loewy-graded AM-exact $H$-comodule algebra generates an ${}_H\M$-module category equivalent to ${}_A\M$ for some homogeneous coideal subalgebra $A\subset H$.
\end{proposition}

%% file: Acknowledge.tex
This work is supported in part by the Erwin and Peggy Kleinfeld Graduate Fellowship and also by the National Science Foundation, Award No. DMS-2303334.